\documentclass[11pt]{article}
\usepackage[utf8]{inputenc}
\usepackage{enumerate}
\usepackage[a4paper]{geometry}
\usepackage{bbm}
\usepackage{bigints}
\usepackage{amsthm}
\usepackage{amssymb}
\usepackage{color}
\usepackage{tikz}
\newtheorem{theorem}{Theorem}
\newtheorem{proposition}{Proposition}
\newtheorem{corollary}{Corollary}
\newtheorem{lemma}{Lemma}

\theoremstyle{definition}
\newtheorem{definition}{Definition}

\theoremstyle{remark}
\newtheorem{remark}{Remark}
\theoremstyle{remark}
\newtheorem{example}{Example}

\def\L{\mathbb{L}}
\def\E{\mathbb{E}}
\def\P{\mathbb{P}}
\def\R{\mathbb{R}}

\def\N{\mathbb{N}}
\def\1{\mathbbm{1}}
\def\d{\partial}
\def\Z{\mathbb{Z}}

\def\Lip{\text{Lip}}

\def\cE{{\cal E}}

\def\cL{\mathcal{L}}
\def\cD{\mathcal{D}}
\def\cC{\mathcal{C}}
\def\cM{\mathcal{M}}
\def\cB{\mathcal{B}}

\def\cW{\mathcal{W}}
\def\Var{\text{Var}}
\def\cP{\mathcal{P}}
\def\Hess{\text{Hess}}
\def\Id{\text{Id}}

\title{Convergence to quasi-stationarity through Poincaré inequalities and Bakry-\'{E}mery criteria}
\author{William Oçafrain$^{1}$}
\date{\today}

\begin{document}

\footnotetext[1]{Institut de Mathématiques, Université de Neuchâtel, Rue Emile-Argand, Neuchâtel,
Suisse-2000 \\
  E-mail: william.ocafrain@hotmail.fr}

\maketitle

\begin{abstract}
This paper aims to provide some tools coming from functional inequalities to deal with quasi-stationarity for absorbed Markov processes. First, it is shown how a Poincaré inequality related to a suitable Doob transform entails exponential convergence of conditioned distributions to a quasi-stationary distribution in total variation and in $1$-Wasserstein distance. A special attention is paid to multi-dimensional diffusion processes, for which the aforementioned Poincaré inequality is implied by an easier-to-check Bakry-\'{E}mery condition depending on the right eigenvector for the sub-Markovian generator, which is not always known. Under additional assumptions on the potential, it is possible to bypass this lack of knowledge showing that exponential quasi-ergodicity is entailed by the classical Bakry-\'{E}mery condition.  
\end{abstract} 

\textit{ Key words :}  Absorbed Markov processes; quasi-stationary distribution; Poincaré inequality; Bakry-\'{E}mery condition; $1$-Wasserstein distance; multi-dimensional diffusion processes.
\bigskip

\textit{ 2010 Mathematics Subject Classification. 60B10; 60F99; 60J25; 60J50; 39B62; 60J60.} 
\bigskip

\section*{Notation}
    For a general metric space $(F,d)$:
    \begin{itemize}
    \item $\cM_1(F)$ : Set of the probability measures defined on $F$.
 \item $\cP_p(F)$ : Set of the probability measures defined on $F$ such that 
$$\int_F d(x_0,x)^p \mu(dx) < + \infty,$$
where $x_0 \in F$ is arbitrary.
    \item $\cB(F)$ : Set of the measurable bounded functions defined on $F$.
    \item $\cB_1(F)$ : Set of the measurable bounded functions defined on $F$ such that $||f||_\infty \leq 1$.
    \item $\L^2(\mu)$ : Set of the functions such that $\int_F |f|^2 d\mu < + \infty$, endowed with the norm
    $$\| \cdot \|_{\L^2(\mu)} : f \mapsto \sqrt{\int_F |f|^2 d\mu}.$$
    \item For any $\mu \in \cM_1(F)$ and $f \in \cB(F)$, 
    $$\mu(f) := \int_F f(x) \mu(dx).$$
    \item For two probability measures $\mu$ and $\nu$, the notation $\mu \ll \nu$ means that there exists a density function $f$ such that
    $$\mu(\cdot) = \int_\cdot f(x) \nu(dx),$$
    and this density function will be denoted by $\frac{d\mu}{d \nu}$. 
\item For any positive measure $\mu$ and any measurable function $f$ such that $\mu(f) < + \infty$, denote $f \circ \mu$ the probability measure defined by
\begin{equation}
    \label{notation}
    f \circ \mu (dx) := \frac{f(x)\mu(dx)}{\mu(f)}.
\end{equation}
\end{itemize}

\section{Introduction}

Consider a time-homogeneous Markov process $(X_t)_{t \geq 0}$ defined on a metric state space $(E \cup \{\d\},d)$, where the element $\d \not \in E$ is a \textit{cemetery point} for the process $X$, which means that 
$$X_t = \d,~~~~\forall t \geq \tau_\d,$$
where $\tau_\d := \inf\{t \geq 0 : X_t = \d\}$ is the hitting time of $\d$.
We associate to the process $(X_t)_{t \geq 0}$ a family of probability measures $(\P_x)_{x \in E}$ such that, for any $x \in E$, $\P_x(X_0 = x) = 1$. For any $\mu \in \cM_1(E \cup \{\d\})$, denote $\P_\mu := \int_E \P_x \mu(dx)$. Then, under $\P_\mu$, the law of $X_0$ is $\mu$. Finally, the expectations $\E_x$ and $\E_\mu$ are  respectively associated to $\P_x$ and $\P_\mu$. Moreover, assume that, for any $x \in E$,
$$\P_x[\tau_\d < + \infty] = 1,~~~~\text{and}~~~~\P_x[\tau_\d > t] > 0,~~\forall t \geq 0.$$

A natural notion to study considering absorbed Markov processes is the notion of \textit{quasi-stationarity}, dealing with the weak convergence of the probability measures $$\P_\mu(X_t \in \cdot | \tau_\d > t)$$ when $t$ goes to infinity. It is well-known that, if such a convergence holds for a given initial law $\mu$, then the limiting probability measure $\alpha$ satisfies 
$$\P_\alpha(X_t \in \cdot | \tau_\d > t) = \alpha,~~~~\forall t \geq 0.$$
Such a probability measure is called a \textit{quasi-stationary distribution} and can be understood as an invariant measure for the semi-flow $(\phi_t)_{t \geq 0}$ defined by
$$\begin{array}{ccccc}
     \phi_t & : & \cM_1(E) & \to & \cM_1(E)  \\
     & & \mu & \mapsto & \P_\mu(X_t \in \cdot | \tau_\d > t),
\end{array}~~~~\forall t \geq 0.$$ 
For a general overview on this theory, we refer the reader to \cite{CMSM,MV2012,vDP2013}, where it is shown that, defining the sub-Markovian semi-group $(P_t)_{t \geq 0}$ as
\begin{equation}
    \label{semi-group}
    P_tf(x) := \E_x(f(X_t)\1_{\tau_\d > t}),~~~~\forall t \geq 0, ~\forall f \in \cB(E), \forall x \in E,
\end{equation}
$\alpha \in \cM_1(E)$ is a quasi-stationary distribution if and only if there exists $\lambda_0 > 0$ such that  
$$\alpha P_t := \P_\alpha(X_t \in \cdot, \tau_\d > t) = e^{-\lambda_0 t} \alpha,~~~~\forall t \geq 0.$$
In other words, quasi-stationary distributions are left eigenvectors for the operators $P_t$, associated to the eigenvalues $e^{-\lambda_0 t}$.
Hence, quasi-stationarity can be dealt with through spectral methods, and existence and uniqueness of quasi-stationary distributions has been shown by this way for several processes, such as discrete-time Markov chains \cite{Darroch1965,SVJ1966}, birth-death processes \cite{cavender1978,KS1991,vD1991} and diffusion processes \cite{CCLMMSM2009,KS2012,littin2012,MSM1994,SE2007}. 

More recently, other methods were developed in order to study quasi-stationarity. These methods aim to obtain exponential convergence towards quasi-stationary distributions for some processes and are based on well-known probabilistic tools coming from the framework without absorption, such as Doeblin's condition or Lyapunov functions (see \cite{MT2012} for an overview on these tools). In particular, in \cite{CV2014},  necessary and sufficient conditions for the uniform-in-law exponential convergence in total variation are provided, where we recall that the \textit{total variation distance} of two probability measures $\mu,\nu$ is defined by
$$\|\mu - \nu \|_{TV} := \sup_{f \in \cB_1(E)} |\mu(f) - \nu(f)|.$$
Since, other papers showed exponential convergences in total variation under weaker assumptions, allowing convergences in total variation holding non-uniformly in the initial measure. In particular, we refer the reader to \cite{CV2017c,V2018} for the study of absorbed Markov processes, and \cite{bansaye2019non,champagnat2019practical,ferre2020more} for the study of general renormalized Feynman-Kac semi-groups. 

For non-absorbed Markov processes, the rate of convergence towards invariant measures can also be studied through functional inequalities, such as Poincaré inequalities. 
A probability measure $\pi$ is said to satisfy a \textit{Poincaré inequality} if there exists a constant $C > 0$ such that, for any $f \in \cD(\cE)$,
\begin{equation}
    \label{ls}
    \Var_\pi(f) \leq - C \int_E f \cL f d\pi,
\end{equation}
where $\Var_\pi(f) := \int_E (f - \pi(f))^2 d\pi$, $\cL$ is a generator which cancels $\pi$, and $\cD(\cE)$ is the set of the measurable functions such that
$$\cE(f,f) := - \int_E f \cL f d \pi$$
is well-defined. We refer the reader to \cite{bakry2013analysis,royer2007initiation} to go further about Poincaré inequalities.   

The inequality \eqref{ls} is actually equivalent to the exponential decay of the \textit{$\chi_2$-divergence} between the semi-group $\mu e^{t \cL}$ and $\pi$, the $\chi_2$-divergence being defined as follows :
$$\chi_2(\mu | \nu) := \left\{ \begin{array}{cc}
     \sqrt{\int_E \left(\frac{d\mu}{d\nu}-1\right)^2 d\nu} & \text{ if } \mu \ll \nu  \\
     + \infty & \text{ otherwise.}
\end{array}
\right.$$
In particular, this implies an exponential decay of the total variation distance between $\mu e^{t \cL}$ and $\pi$ when the quantity $\chi_2(\mu | \pi)$ is finite. 

In the literature, some papers dealing with the use of Poincaré inequalities for quasi-stationarity have been already written, in particular for Markov processes living on discrete state spaces (\cite{CCM2016,diaconis2019analytic,DM2015}). However, the proofs provided by these papers strongly rely on the discrete aspect of the state space, and are therefore hardly applicable for processes living on continuous state space, such as diffusions processes. Our aim will be therefore to show how to use such inequalities to get exponential convergence towards quasi-stationarity for such processes. In particular, the convergence in total variation will be studied, as well as the convergence in \textit{$1$-Wasserstein distance}, which is defined as 
$$\cW_1(\mu, \nu) := \inf_{(X,Y) \in \Pi(\mu,\nu)} \E[d(X,Y)],~~~~\forall \mu, \nu \in \cP_1(E),$$
where $\Pi(\mu,\nu)$ is the set of all the couplings $(X,Y)$ such that the law of $X$ (respectively $Y$) is $\mu$ (respectively $\nu$). We refer to Theorem \ref{thm-poin} in Section \ref{section-poin} for the general statement and Corollary \ref{cor1} for the convergence in $1$-Wasserstein distance.

In the third and last section, we will be more particularly interested in quasi-stationarity for diffusion processes $(X_t)_{t \geq 0}$ living on a domain $D \subset \R^d$, absorbed at the boundary $\d D$, and satisfying on $D$ the stochastic differential equation
\begin{equation}
    \label{sde}
    dX_t = \sqrt{2} dB_t - \nabla V(X_t)dt,~~~~X_t \in D,
\end{equation}
where $(B_t)_{t \geq 0}$ is a $d$-dimensional Brownian motion and $V$ is a $\cC^2$-function on $\R^d$. In the non-absorbed framework, it is well-known that the reversible probability measure $$\gamma(dx) := Z^{-1} e^{-V(x)}dx$$ ($Z$ is the renormalization constant) satisfies a Poincaré inequality when the condition
\begin{equation}
    \label{be}
    \Hess ~ V \geq \kappa \Id
\end{equation}
is satisfied for a given $\kappa > 0$. This last result is a consequence from the one shown by Bakry and \'{E}mery in \cite{bakry1985diffusions} and the condition \eqref{be} is usually called the \textit{Bakry-\'{E}mery condition} or \textit{curvature-dimension condition}. In particular, under \eqref{be}, the diffusion process $(X_t)_{t \geq 0}$ converges towards $\gamma$  in total variation and in $1$-Wasserstein distance.
Our goal is therefore to recover this property of convergence in the quasi-stationary framework through a condition similar to \eqref{be}. More precisely, the following result is obtained in Section \ref{section-be} :
\begin{theorem}
\label{premiere}
\begin{itemize}
    \item Assume that there exists $\eta$ positive on $D$, vanishing on $\d D$ such that $\gamma(\eta^2) < + \infty$ and there exists $\lambda_0 > 0$ such that
    $$\Delta \eta(x) - \nabla V(x) \cdot \nabla f(x) = - \lambda_0 \eta(x),~~~~\forall x \in D,$$
    \item and assume that there exists $\kappa > 0$ such that 
    $$\Hess [V - 2 \log(\eta)] \geq \kappa \Id.$$
\end{itemize}
Then there exists $C > 0$ such that, for any $\mu \in \cM_1(D)$ and $t \geq 0$,
    \begin{equation}
        \label{inequality}
        \| \P_\mu[X_t \in \cdot | \tau_\d > t] - \eta \circ \gamma \|_{TV} \leq C \chi_2(\eta \circ \mu | \eta^2 \circ \gamma) e^{-\kappa t},
    \end{equation}
    where we recall that the notation $f \circ \mu$ is defined previously in \eqref{notation} in Notation.
    
    Moreover, if $\int_D (1+|x|)^2 e^{-V(x)}dx < + \infty$, the inequality $\eqref{inequality}$ holds in $1$-Wasserstein distance for $t$ large enough.
\end{theorem}
A more specific study will focus on multi-dimensional diffusion processes living on $D = (0,+\infty)^d$ and absorbed when one component is $0$. In this particular case, and assuming moreover that $V$ can be expressed as 
$$V(x_1, \ldots, x_d) = \sum_{i=1}^d V_i(x_i),~~~~\forall ~(x_1, \ldots, x_d) \in D,$$
where, for all $i$, $V_i$ are $\cC^2$-functions, one has the following result :  
\begin{theorem}
\label{last}
If 
$$\Hess~V \geq \kappa \Id,$$
then there exists a quasi-stationary distribution $\alpha = \alpha_1 \otimes \cdots \otimes \alpha_d \in \cM_1(D)$ and $C_d > 0$ (depending on the dimension $d$) such that, for any $\mu \in \cM_1(D)$ and $t$ large enough,
$$\cW_1(\P_\mu[X_t \in \cdot | \tau_\d > t], \alpha) \leq C_d \chi_2(\eta \circ \mu | \eta \circ \alpha) e^{-\kappa t},$$
where $\eta := \frac{d \alpha}{d \gamma}$.
If moreover $\mu = \mu_1 \otimes \cdots \otimes \mu_d$, there exists a constant $C > 0$, which does not depend on $d$, such that, for $t$ large enough,
$$\cW_1(\P_\mu[X_t \in \cdot | \tau_\d > t], \alpha) \leq C \left[ \sum_{i=1}^d \chi_2(\eta_i \circ \mu_i | \eta_i \circ \alpha_i)\right] e^{-\kappa t},$$
where $\eta_i := \frac{d\alpha_i}{d\gamma_i}$. 
\end{theorem}
This theorem is further referenced as Theorem 6 in Subsection 3.3.3.
A particular attention will be paid on processes coming down from infinity, for which it will be shown that the rate of convergence $\kappa$ provided by the Bakry-\'{E}mery condition \eqref{be} can actually be bettered (see Theorems \ref{thm-cdfi} and \ref{cdfimulti}). 

\section{Exponential convergence to quasi-stationarity through a Poincaré inequality}
\label{section-poin}

\subsection{Main result}

Let $(X_t)_{t \geq 0}$ be a Markov process absorbed at a cemetery point $\d$, and let $(P_t)_{t \geq 0}$ be the sub-Markovian semi-group defined in \eqref{semi-group}. Denote by 
\begin{equation} \label{domain-of-definition} \cD(\cL) := \left\{f : \lim_{t \to 0} \frac{P_tf(x) - f(x) }{t} \text{ exists for any } x\right\},\end{equation}
and define the generator $\cL$ as
 \begin{equation} \label{generator}\cL f(x) := \lim_{t \to 0} \frac{P_tf(x) - f(x) }{t},~~~~\forall x \in E, \forall f \in \cD(\cL).\end{equation}
Now, let us state the following theorem : 

\begin{theorem}
\label{thm-poin}
Assume that
\begin{enumerate}[({P}1)]
    \item there exists a quasi-stationary distribution $\alpha \in \cM_1(E)$ satisfying 
    $$\alpha P_t = e^{-\lambda_0 t} \alpha,~~~~\forall t \geq 0,$$
    with $\lambda_0 > 0$, and an eigenfunction $\eta$ positive on $E$ such that $\alpha(\eta) = 1$ and
    $$P_t \eta(x) = e^{-\lambda_0 t} \eta(x),~~~~\forall x \in E, \forall t \geq 0;$$
    \item there exists $C_P \in (0,+\infty)$ such that
    \begin{equation}
    \label{poincare}
        \Var_{\eta \circ \alpha}(f) \leq - C_P \left[\int_E f (\lambda_0 f \eta + \cL(f \eta)) d \alpha \right],
    \end{equation}
    for any measurable function $f$ such that $\lambda_0 \int_E f^2 d(\eta \circ \alpha) +  \int_E f \cL(f \eta) d \alpha$ is well-defined;
\item and there exists a function $\psi : E \to [1,+ \infty)$ such that $$\alpha(\psi) < + \infty, ~~~~\text{ and }~~~~ \alpha(\psi^2/\eta) < + \infty.$$ 
\end{enumerate}
Then, for any $\mu \in \cM_1(E)$, there exists $t_\mu$ such that, for any $t \geq t_\mu$,
\begin{equation}
\label{expo-conv}
\sup_{|f| \leq \psi} \left| \E_\mu[f(X_t) | \tau_\d > t] - \alpha(f) \right| \leq C_\psi \chi_2(\eta \circ \mu | \eta \circ \alpha) e^{-\frac{t}{C_P}},\end{equation}
where 
$$C_\psi = (a+b \alpha(\psi)) \alpha(\psi^2/\eta)^{1/2},$$
with some positive constants $a,b$. 
\end{theorem}

Before proving this theorem, let us do some important remarks:

\begin{remark}
\label{remark}
(P1) is satisfied under the Lyapunov conditions (E) and (F) presented in \cite{CV2017c}, or under the conditional Doeblin's conditions (A) in \cite{CV2014}. In particular, the results presented further in Subsection \ref{dif}, dealing with quasi-stationarity for processes following stochastic differential equations like
$$dX_t = \sqrt{2} dB_t - \nabla V(X_t)dt,$$
rely a lot on these assumptions, allowing to state easy-to-check criteria only based on the potential $V$ (such as \eqref{c} in Remark \ref{r}). For other criteria related to other types of processes (such as diffusion processes with multiplicative noise, birth-and-death processes, Markov chains,...), we refer the reader to \cite{CV2017c}.

It is possible to state even less restrictive assumptions than \cite[Assumptions (F)]{CV2017c}, also based on Lyapunov functions, entailing the existence of a quasi-stationary distribution (see for example \cite[Theorem 4.2.]{collet2009quasistationary} or \cite[Theorem 7]{BCOV2021}), but without ensuring the existence of an eigenfunction $\eta$. To obtain such an eigenfunction, it is quite usual to use Krein-Rutman's theorem, once the compacity of the operators $(P_t)_{t \geq 0}$ or related operators is known. Finally, note that the existence of $\alpha$ and $\eta$ does not ensure in general that $\alpha(\eta) < + \infty$. For example, the one-dimensional Brownian motion with drift $X_t = B_t - r t$ ($r > 0$) absorbed at $0$ admits $\alpha(dx) = r^2 x e^{-rx}dx$ and $\eta(x) = x e^{rx}$, so that $\alpha(\eta) = + \infty$.
\end{remark}   

\begin{remark}
\label{petcou}
It is also important to note that the Poincaré constant $C_P$ fundamentally depends on the survival state space $E$. We refer the reader to the subsection \ref{ex}, in particular the example of a Brownian motion living in the hypercube $C_N : (-N,N)^d \subset \R^d$, for which the Poincaré constant increases as $N^2$ when $N$ increases.
\end{remark}   

\begin{remark}
\label{domain-of-attraction}
For several processes, it is quite usual to have  $\chi_2(\eta \circ \mu | \eta \circ \alpha) = + \infty$ when the initial law is a Dirac measure $\delta_x$. In the most of the cases (see for instance the two examples provided in Subsection 3.2), considering a state $x \in E$, there exists a time $t_0 > 0$ such that
$$\chi_2(\eta\circ \phi_{t_0}(\delta_x) | \eta \circ \alpha) < + \infty.$$
Hence, using the property of semi-flow of $(\phi_t)_{t \geq 0}$ (i.e. $\phi_{t+s} = \phi_t \circ \phi_s$ for all $s,t \geq 0$), the previous theorem implies that there exists $t_{\phi_{t_0}(\delta_x)}$ such that, for any $t \geq t_0 + t_{\phi_{t_0}(\delta_x)}$,
\begin{align*}
\sup_{|f| \leq \psi} \left| \E_x[f(X_t) | \tau_\d > t] - \alpha(f) \right| &= \sup_{|f| \leq \psi} \left| \E_{\phi_{t_0}(\delta_x)}[f(X_{t-t_0}) | \tau_\d > t-t_0] - \alpha(f) \right| \\
&\leq C_\psi \chi_2(\eta\circ \phi_{t_0}(\delta_x) | \eta \circ \alpha) e^{-\frac{(t-t_0)}{C_P}}.
\end{align*}
In other terms, the set of all the measures such that there exists $t_0 \geq 0$ such that $\chi_2(\eta\circ \phi_{t_0}(\mu) | \eta \circ \alpha) < + \infty$ is included in the \textit{domain of attraction of $\alpha$}, denoted by $\cD(\alpha)$, that is the set of the initial measures such that the weak convergence $\P_\mu[X_t \in \cdot | \tau_\d > t] \underset{t \to \infty}{\longrightarrow} \alpha$ holds. We refer the reader to Subsection \ref{DoA} for a deepening on the study of domain of attraction. 
\end{remark} 

\begin{remark}
Because of the condition $\psi \geq 1$, the distance $\sup_{|f| \leq \psi} \left| \E_\mu[f(X_t) | \tau_\d > t] - \alpha(f) \right|$ is actually stronger than the total variation distance.
In particular, Theorem \ref{thm-poin} implies that there exists $C > 0$ such that, for any $\mu \in \cM_1(E)$ and $t \geq 0$, 
$$\| \P_\mu[X_t \in \cdot | \tau_\d > t] - \alpha \|_{TV} \leq C \chi_2(\eta \circ \mu | \eta \circ \alpha) e^{-\frac{t}{C_P}}.$$
Hence, this theorem allows to obtain a result analogous to the ones obtained by Champagnat and Villemonais in \cite{CV2014,CV2017c}. However, contrary to their results, the upper bound could be small if the initial measure is close enough to the quasi-stationary distribution $\alpha$ (it is even equal to $0$ for $\mu = \alpha$). Moreover, Theorem \ref{thm-poin} allows to obtain a convergence in $1$-Wasserstein distance, as stated by the following corollary :
\end{remark}
\begin{corollary}
\label{cor1}
If the assumptions $(P3)$ holds for $$\psi : x \mapsto 1 + d(x,x_0)$$ for a given $x_0 \in E$, then, for any $\mu \in \cM_1(E)$, there exists $t_\mu$ such that, for any $t \geq t_\mu$,
$$\cW_1(\P_\mu[X_t \in \cdot | \tau_\d > t],\alpha) \leq C_\psi \chi_2(\eta \circ \mu | \eta \circ \alpha) e^{-\frac{t}{C_P}}.$$

\end{corollary}

\begin{proof}
By the dual formula for the $1$-Wasserstein distance (for example see \cite{V2009}), for any probability measures $\mu$ and $\nu$ in $\cP_1(E)$, one has
\begin{equation}
\label{dual-formula}
    \cW_1(\mu,\nu) = \sup_{f ~1-\Lip} |\mu(f) - \nu(f)| = \sup_{f \in \cC} |\mu(f)-\nu(f)|,
\end{equation}
where 
$$\cC := \{f ~1-\Lip : f(x_0) = 1\}.$$
Thus, any function $f$ belonging to $\cC$ satisfies
$$|f(x)| \leq 1 + d(x,x_0) = \psi(x),~~~~\forall x \in E.$$
Hence, by Theorem \ref{thm-poin}, there exists $t_\mu$ such that, for any $t \geq t_\mu$,
\begin{align*}
    \cW_1(\P_\mu[X_t \in \cdot | \tau_\d > t],\alpha) &\leq \sup_{|f| \leq \psi} \left| \E_\mu[f(X_t) | \tau_\d > t] - \alpha(f) \right|\\ &\leq C_\psi \chi_2(\eta \circ \mu | \eta \circ \alpha) e^{-\frac{t}{C_P}}.
\end{align*}
\end{proof}

Now, let us tackle the proof of Theorem \ref{thm-poin} :

\begin{proof}[Proof of Theorem \ref{thm-poin}] First, remark that if $\chi_2(\eta \circ \mu | \eta \circ \alpha) = + \infty$, the inequality \eqref{expo-conv} is trivially satisfied. So, from now on, we will only consider initial measure such that 
$$\chi_2(\eta \circ \mu | \eta \circ \alpha) < + \infty.$$
The proof is divided into two steps. \\\\
\underline{\textit{First step: When $\alpha(\psi^2/\eta) \chi^2_2(\eta \circ \mu | \eta \circ \alpha) < 0.9$\footnote{2}.}} \footnotetext[1]{The choice of the value 0.9 is totally arbitrary, any value smaller than 1 is suitable for the proof.} \\\\
Let $\mu \in \cM_1(E)$ satisfying $\alpha(\psi^2/\eta) \chi^2_2(\eta \circ \mu | \eta \circ \alpha) < 0.9$. Denote by $(\tilde{P}_t)_{t \geq 0}$ the Markovian semi-group defined by
$$\tilde{P}_tf(x) := e^{\lambda_0 t} \frac{P_t[f \eta](x)}{\eta(x)},$$
where we recall that $\lambda_0$ and $\eta$ are such that, for any $x \in E$ and $t \geq 0$,
$$P_t \eta(x) = e^{- \lambda_0 t} \eta(x).$$
Then, since $\alpha$ is a quasi-stationary distribution for $(X_t)_{t \geq 0}$, the probability measure $\beta(dx) := \eta(x) \alpha(dx)$ is an invariant measure for $(\tilde{P}_t)_{t \geq 0}$. Moreover, denoting by $\tilde{\cL}$ the generator of $(\tilde{P}_t)_{t \geq 0}$, then, for any measurable $f$ such that $f \eta \in \cD(\cL)$ and for any $x \in E$,
$$\tilde{\cL}f(x) = \lambda_0 f(x) + \frac{\cL(f \eta)(x)}{\eta(x)}.$$
This equality comes from the equality $\tilde{\cL} f(x) := \frac{d \tilde{P}_t f(x)}{dt}\big|_{t=0}$.

Then, the Poincaré inequality \eqref{poincare} can be written as follows : 
$$\Var_\beta(f) \leq -C_P \int_E f \tilde{\cL}f d\beta.$$
In other words, the inequality \eqref{poincare} is the Poincaré inequality for the Markovian semi-group $(\tilde{P}_t)_{t \geq 0}$. Then it is well-known that it is equivalent to : for any probability measure $\nu$ on $E$ and $t \geq 0$,
\begin{equation}
    \label{decay}
    \chi^2_2(\nu \tilde{P}_t | \beta) \leq e^{-\frac{2t}{C_P}} \chi^2_2(\nu | \beta).
\end{equation}
Now, let us define, for any $f \in \cB(E)$, $t \geq 0$ and $x \in E$, 
$$Q_t[f](x) := \tilde{P}_t[f/\eta](x) = \frac{e^{\lambda_0 t}}{\eta(x)} P_t[f](x).$$
Since $\alpha(\psi^2/\eta) < \infty$ by the third assumption, one has, for any measurable function such that $|f| \leq \psi$, 
$$||f/\eta||_{\L^2(\beta)} \leq \alpha(\psi^2/\eta) < \infty.$$
In particular,  for any measurable function $f$ such that $|f| \leq \psi$, $t \geq 0$ and $\nu \in \cM_1(E)$ 
$$\left|\nu Q_t f - \alpha(f)\right|^2 = \left[ \nu \tilde{P}_t [f/\eta] - \beta[f/\eta]\right]^2 \leq \alpha(\psi^2/\eta) e^{-\frac{2t}{C_P}} \chi^2_2(\mu | \beta),$$
where the following equality is used : $\forall \nu_1, \nu_2 \in \cM_1(E)$,
$$\chi^2_2(\nu_1 | \nu_2) = \sup_{\|f\|_{\L^2(\nu_2)} \leq 1} \left| \nu_1(f) - \nu_2(f) \right|^2.$$
As a result,  for any $t \geq 0$ and any $\nu \in \cM_1(E)$,
$$\sup_{|f| \leq \psi} |\nu Q_t[f] - \alpha(f)| \leq  [\alpha(\psi^2/\eta)\chi^2_2(\nu | \beta)]^{1/2} e^{- \frac{t}{C_P}}.$$
Now note that, for any $t \geq 0$ and any measurable function $f$, 
\begin{align}
\E_\mu[f(X_t) | \tau_\d > t] &= \frac{\int_E P_t[f](x) \mu(dx)}{\int_E P_t[\1_E](x) \mu(dx)} \notag \\
&=  \frac{\int_E \frac{e^{\lambda_0 t}}{\eta(x)}P_t[f](x) \eta(x) \mu(dx)}{ \int_E \frac{e^{\lambda_0 t}}{\eta(x)} P_t[\1_E](x)  \eta(x)\mu(dx)} \notag \\
&= \frac{\int_E Q_t[f](x) \eta(x) \mu(dx)}{\int_E Q_t[\1_{E}](x) \eta(x)\mu(dx)} \notag \\
&= \frac{(\eta \circ \mu) Q_t [f]}{(\eta \circ \mu) Q_t [\1_{E}]}, \label{ef2}
\end{align}
As a result, since $\alpha(\psi^2/\eta) \chi^2_2(\eta \circ \mu | \beta) < 0.9$, for any $t \geq 0$, 
\begin{equation}
\label{encadrement}
   \frac{\alpha(f) -   [\alpha(\psi^2/\eta)\chi^2_2(\eta \circ \mu | \beta)]^{1/2} e^{-\frac{t}{C_P}}}{1 +  [\alpha(\psi^2/\eta)\chi^2_2(\eta \circ \mu | \beta)]^{1/2} e^{-\frac{t}{C_P}}} \leq \E_\mu(f(X_t)|\tau_\d > t) \leq \frac{\alpha(f) +  [\alpha(\psi^2/\eta)\chi^2_2(\eta \circ \mu | \beta)]^{1/2} e^{-\frac{t}{C_P}}}{1 -  [\alpha(\psi^2/\eta)\chi^2_2(\eta \circ \mu | \beta)]^{1/2} e^{-\frac{t}{C_P}}}.
\end{equation}
For any $t \geq 0$,
\begin{multline*}
\frac{\alpha(f) +  [\alpha(\psi^2/\eta)\chi^2_2(\eta \circ \mu | \beta)]^{1/2} e^{-\frac{t}{C_P}}}{1 -  [\alpha(\psi^2/\eta)\chi^2_2(\eta \circ \mu | \beta)]^{1/2} e^{-\frac{t}{C_P}}} \\= \left(\alpha(f) +  [\alpha(\psi^2/\eta)\chi^2_2(\eta \circ \mu | \beta)]^{1/2} e^{-\frac{t}{C_P}}\right)\left[ 1 + \frac{  [\alpha(\psi^2/\eta)\chi^2_2(\eta \circ \mu | \beta)]^{1/2} e^{-\frac{t}{C_P}}}{1-  [\alpha(\psi^2/\eta)\chi^2_2(\eta \circ \mu | \beta)]^{1/2} e^{-\frac{t}{C_P}}}\right] \\
\leq\alpha(f) +  [\alpha(\psi^2/\eta)\chi^2_2(\eta \circ \mu | \beta)]^{1/2} e^{-\frac{t}{C_P}} + \left(\alpha(\psi) + 1 \right) \frac{  [\alpha(\psi^2/\eta)\chi^2_2(\eta \circ \mu | \beta)]^{1/2} e^{-\frac{t}{C_P}} }{1-  [\alpha(\psi^2/\eta)\chi^2_2(\eta \circ \mu | \beta)]^{1/2} e^{-\frac{t}{C_P}} } \\
\leq \alpha(f) + \left(1 + \frac{\alpha(\psi) + 1}{1-  [\alpha(\psi^2/\eta)\chi^2_2(\eta \circ \mu | \beta)]^{1/2}}\right)  [\alpha(\psi^2/\eta)\chi^2_2(\eta \circ \mu | \beta)]^{1/2} e^{-\frac{t}{C_P}}.
\end{multline*}
In a same way, one can prove that, for any $t \geq 0$,
$$\alpha(f) - (2+\alpha(\psi))   [\alpha(\psi^2/\eta)\chi^2_2(\eta \circ \mu | \beta)]^{1/2} e^{-\frac{t}{C_P}} \leq   \frac{\alpha(f) -   [\alpha(\psi^2/\eta)\chi^2_2(\eta \circ \mu | \beta)]^{1/2} e^{-\frac{t}{C_P}}}{1 +  [\alpha(\psi^2/\eta)\chi^2_2(\eta \circ \mu | \beta)]^{1/2} e^{-\frac{t}{C_P}}} .$$
As a result, using \eqref{encadrement}, for any $t \geq 0$,
\begin{multline*}
\sup_{|f| \leq \psi} \left| \E_\mu(f(X_t)|\tau_\d > t) - \alpha(f) \right|\\ \leq  \left[\left(1 + \frac{\alpha(\psi) + 1}{1-  [\alpha(\psi^2/\eta)\chi^2_2(\eta \circ \mu | \beta)]^{1/2}}\right) \lor  (2+\alpha(\psi))\right] [\alpha(\psi^2/\eta)\chi^2_2(\eta \circ \mu | \beta)]^{1/2} e^{- \frac{t}{C_P}} \\
\leq (a+b \alpha(\psi)) (\alpha(\psi^2/\eta)\chi^2_2(\eta \circ \mu | \beta))^{1/2} e^{-\frac{t}{C_P}},\end{multline*}
setting
$$a := 1 + \frac{1}{1- \sqrt{0.9}},~~~~~~~~\text{       and      }~~~~~~~~b := \frac{1}{1- \sqrt{0.9}}.$$
\underline{\textit{Second step : Conclusion.}}\\\\
Now let $\mu \in \cM_1(E)$ such that $\chi_2(\eta \circ \mu | \beta) < + \infty$. Recalling the notation
$$\phi_t(\mu) := \P_\mu[X_t \in \cdot | \tau_\d > t],$$  one has the following lemma, whose the proof is postponed after the end of this proof.
\begin{lemma}
\label{checkpoint}
For any $t \geq 0$ and $\mu \in \cM_1(E)$,
$$\eta\circ \phi_t(\mu) = (\eta \circ \mu)\tilde{P}_t$$
\end{lemma}
Then, using Lemma \ref{checkpoint} and the inequality \eqref{decay}, one has for any $t \geq 0$,
$$\chi^2_2(\eta\circ \phi_t(\mu) | \beta) = \chi^2_2((\eta \circ \mu)\tilde{P}_t | \beta) \leq e^{-\frac{2t}{C_P}} \chi^2_2(\eta \circ \mu| \beta).$$
In particular, there exists $t_\mu \geq 0$ such that, for any $t \geq t_\mu$,
$$\alpha(\psi^2/\eta) \chi^2_2(\eta\circ \phi_t(\mu) | \beta) < 0.9.$$
Hence, applying what we obtained at the first step, one has, for any $t \geq t_\mu$,
\begin{align*}
    \sup_{|f| \leq \psi} | \E_\mu[f(X_t)|\tau_\d > t] - \alpha(f) | &\leq (a+b \alpha(\psi)) [\alpha(\psi^2/\eta)\chi^2_2(\eta\circ \phi_{t_\mu}(\mu) | \beta)]^{1/2} e^{-\frac{t-t_\mu}{C_P}} \\
    &\leq (a+b \alpha(\psi)) [\alpha(\psi^2/\eta)\chi^2_2(\eta \circ \mu | \beta)]^{1/2} e^{-\frac{t}{C_P}},
\end{align*}
which concludes the proof.
\end{proof}
Now, let us prove Lemma \ref{checkpoint}.
\begin{proof}[Proof of Lemma \ref{checkpoint}]
For any $t \geq 0$, $\mu \in \cM_1(E)$ and for any measurable function $f$,
\begin{align*}
    \eta\circ \phi_t(\mu)(f) &= \frac{\phi_t(\mu)(f \eta)}{\phi_t(\mu)(\eta)} \\
    &= \frac{\mu P_t [f \eta]}{\mu P_t[\eta]} \\
    &= \frac{e^{\lambda_0 t} \mu P_t [f \eta]}{\mu(\eta)} \\
    &= \frac{1}{\mu(\eta)} \int_E \tilde{P}_t[f](x) \eta(x) \mu(dx) \\
    &= (\eta \circ \mu) \tilde{P}_t[f],
\end{align*}
where the equality $P_t[\eta](x) = e^{- \lambda_0 t} \eta(x)$, for any $t \geq 0$ and $x \in E$, was used. 
\end{proof}

\begin{remark}
By the tensorization property of Poincaré inequalities (see \cite[Proposition 4.3.1]{bakry2013analysis}), the Poincaré constant $C_P$ does not depend on the dimension when the state space is a tensorial space. As a result, contrary to the technics using Lyapunov functions or minorization properties, the previous theorem provides in such cases a rate of convergence which does not explode in high dimension (as soon as the state space $E$ is the product space of one-dimensional spaces $E_i$). \end{remark}

\begin{remark} 
\label{polynomial}
In the same manner, subgeometrical convergences to quasi-stationarity can be proved replacing the conditions (P2) and (LS2) by weaker functional inequalities, such as Nash inequalities or weak Poincaré inequalities (see \cite{liggett1991,rockner2001weak}). This method does not allow however to cover all the processes having this property of subgeometrical convergence (see for example \cite{occafrain2020polynomial} where the Doob transform is not ergodic).   
\end{remark}  

\begin{remark}
As stated in Corollary \ref{cor1}, the previous method using the Doob transform $\tilde{P}_t$ allows to get convergence in $1$-Wasserstein distance through a Poincaré inequality. A natural question is therefore if one can use the logarithmic Sobolev inequality
$$\int_E f^2 \log\left(\frac{f^2}{\|f\|_{\L^2(\beta)}}\right) d\beta \leq -C_{LS} \int_E f \tilde{\cL} f d \beta~~~~\text{(with $C_{LS} > 0$)}$$ 
to deal with the convergence in $p$-Wasserstein distance,  which is defined by 
$$\cW_p(\mu, \nu) := \inf_{(X,Y) \in \Pi(\mu,\nu)} \E[d(X,Y)^p]^{1/p},~~~~\forall \mu, \nu \in \cP_p(E).$$
By the same methodology and using that, for any $\mu,\nu \in \cM_1(E)$,
$$H(\mu | \nu) = \sup_{f \in \cB(E)} \{ \mu(f) - \log(\nu(e^f))\},$$
where $H(\mu | \nu) := \int_E \log\left(\frac{d\mu}{d\nu}\right) d\nu$ (when $\mu \ll \nu$) is the \textit{entropy}, one obtains the one-sided estimate
$$\E_\mu[f(X_t)|\tau_\d > t] \leq \log(\beta(e^{f/\eta})) + C H(\eta \circ \mu | \beta) e^{-\frac{t}{C_{LS}}}.$$
This estimate is unfortunately not sharp enough, since $\log(\beta(e^{f/\eta})) \geq \alpha(f)$, and the convergence in $\cW_p$ for general $p$ still remains an open question.
\end{remark}

\subsection{On the domain of attraction of $\alpha$}
\label{DoA}

This subsection follows Remark \ref{domain-of-attraction}. In this remark, it was pointed that a consequence of Theorem \ref{thm-poin} is the following inclusion: 
$$\left\{\mu \in \cM_1(E) : \exists t_0 \geq 0, \chi_2(\eta \circ \phi_{t_0}(\mu) | \beta) < + \infty\right\} \subset \cD(\alpha),$$
where $\cD(\alpha)$ is the domain of attraction of $\alpha$, that is the set of initial measures such that the convergence of $\P_\mu[X_t \in \cdot | \tau_\d > t]$ to $\alpha$ holds. 

The aim of this subsection is to go a bit further on this point. Let us first state and prove the following proposition, which gives an expression of $\chi_2(\eta \circ \mu | \eta \circ \alpha)$ revealing the density function $\frac{d \mu}{d \alpha}$:

\begin{proposition}
\label{domain-of-attraction-lemma}
For any $\mu \in \cM_1(E)$ such that $\frac{d \mu}{d \alpha}$ exists,
$$\chi^2_2(\eta \circ \mu | \eta \circ \alpha) = \frac{1}{\mu(\eta)^2} \int_E \frac{d \mu}{d \alpha}(y) \times \eta(y) \mu(dy) - 1.$$
In particular, $\chi_2(\eta \circ \mu | \eta \circ \alpha) < + \infty$ if and only if $\mu(\eta) > 0$ and $\int_E \frac{d \mu}{d \alpha} \eta d\mu < + \infty$.
\end{proposition}

\begin{proof}
First of all, recalling that
\begin{equation} \label{def}\eta \circ \mu(dx) = \frac{\eta(x) \mu(dx)}{\mu(\eta)},~~~~\eta \circ \alpha(dx) = \eta(x) \alpha(dx),\end{equation}
$\frac{d \mu}{d \alpha}$ exists if and only if $\frac{d (\eta \circ \mu)}{d (\eta \circ \alpha)}$ exists. Then,
by definition of $\chi_2$,
$$\chi_2^2(\eta \circ \mu | \eta \circ \alpha) = \int_E \left(\frac{d (\eta \circ \mu)}{d (\eta \circ \alpha)} - 1 \right)^2 d(\eta \circ \alpha) = \int_E \left(\frac{d (\eta \circ \mu)}{d (\eta \circ \alpha)} \right)^2 d(\eta \circ \alpha) - 1.$$
Using again \eqref{def}, 
$$\int_E \left(\frac{d (\eta \circ \mu)}{d (\eta \circ \alpha)} \right)^2 d(\eta \circ \alpha) = \int_E \frac{d (\eta \circ \mu)}{d (\eta \circ \alpha)} d(\eta \circ \mu) = \frac{1}{\mu(\eta)^2} \int_E \frac{d \mu}{d \alpha}(y) \times \eta(y) \mu(dy).$$
Thus,
$$\chi^2_2(\eta \circ \mu | \eta \circ \alpha) = \frac{1}{\mu(\eta)^2} \int_E \frac{d \mu}{d \alpha}(y) \times \eta(y) \mu(dy) - 1,$$
which concludes the proof.
\end{proof}

Proposition \ref{domain-of-attraction-lemma} entails the following corollary:

\begin{corollary}
\label{attraction}
Let $\mu \in \cM_1(E)$ such that $\mu(\eta) > 0$. If there exists $t_0 \geq 0$ such that $\frac{d \phi_{t_0}(\mu)}{d \alpha}$ exists and  
$$\int_E \frac{d \phi_{t_0}(\mu)}{d \alpha} \eta d\phi_{t_0}(\mu) < + \infty,$$
then $\chi_2(\eta \circ \phi_{t_0}(\mu) | \eta \circ \alpha) < + \infty$. 
\end{corollary}

\begin{proof}
In order to use Proposition \ref{domain-of-attraction-lemma}, we have just to ensure that 
$$\phi_{t_0}(\mu)(\eta) > 0.$$
However, one has 
$$\phi_{t_0}(\mu)(\eta) = \frac{\mu P_{t_0} \eta}{\mu P_{t_0} \1_E} = \frac{e^{-\lambda_0 t_0} \mu(\eta)}{\P_\mu[\tau_\d > t_0]} > 0,$$
which concludes the proof.
\end{proof}


As a consequence, Corollary \ref{attraction} entails, for a process satisfying $(P1) - (P3)$, that if $\mu(\eta) > 0$ and there exists $t_0 \geq 0$ such that $\int_E \frac{d \phi_{t_0}(\mu)}{d \alpha} \eta d\phi_{t_0}(\mu) < + \infty$, then $\P_\mu[X_t \in \cdot | \tau_\d > t]$ converges to $\alpha$ exponentially fast at rate $1/C_P$. This property will be then used in Subsection \ref{ex} to show, in both examples presented in this subsection, that there is an exponential decay at rate $1/C_P$ when the initial measure is a Dirac measure, even if such a measure does not admit any density function with respect to $\alpha$.  

Also, Proposition \ref{domain-of-attraction-lemma} entails another interesting corollary, when $\eta$ is known to be upper-bounded:

\begin{corollary}
Let $\mu \in \cM_1(E)$. If $\eta$ is upper-bounded, $\mu(\eta) > 0$ and $\chi_2(\mu | \alpha) < + \infty$, then $$\chi_2(\eta \circ \mu | \eta \circ \mu) < + \infty.$$  
\end{corollary}

\begin{proof}
If $\eta$ is upper-bounded, for any $\mu \in \cM_1(E)$ absolutely continuous with respect to $\alpha$,
$$\int_E \frac{d\mu}{d\alpha} \eta d\mu \leq \|\eta\|_\infty \left(1 + \chi_2^2(\mu | \alpha) \right).$$
\end{proof}

\section{Bakry-\'{E}mery condition and quasi-stationarity : application to diffusion processes}
\label{section-be}

In a practical way, Theorem \ref{thm-poin} is hardly useable because the expressions of the quasi-stationary distribution $\alpha$ and the eigenfunction $\eta$ are scarcely explicitly known, so the conditions (P2)-(P3) cannot be checked. In this section, diffusion processes will be only dealt with and easy-to-check assumptions will be given.  

In all what follows, the space $\R^d$ will be endowed with the $L^1$-distance 
\begin{equation}
    \label{distance}
    d(x,y) := \sum_{i=1}^d |x_i - y_i|
\end{equation}
for any $x = (x_i)_{i=1, \ldots, d}$ and $y = (y_i)_{i=1, \ldots, d}$. In particular, this distance will be implicitly used for the definition of $\cW_1$.

Let $D \subset \R^d$ be an open subset of $\R^d$ and $\d D$ its boundary. Let $(X_t)_{t \geq 0}$ be the absorbed diffusion process following
\begin{equation}
\label{sed2}
dX_t = \sqrt{2} dB_t - \nabla V(X_t) dt,~~~~~~X_t \in D,
\end{equation}
with a $d$-dimensional Brownian motion $(B_t)_{t \geq 0}$ and $V \in \cC^2(\R^d)$, and absorbed when $t \geq \tau_\d$, where $$\tau_\d := \inf\{t \geq 0 : X_t \in \d D\}.$$ In order to keep the same notation as the ones in Section \ref{section-poin}, let $(P_t)_{t \geq 0}$ the sub-Markovian semi-group defined in \eqref{semi-group}, $\cL$ the sub-Markovian generator defined in \eqref{generator} and $\cD(\cL)$ the associated domain of definition defined in \eqref{domain-of-definition}. In particular, any function $f \in \cC^2(D)$ with compact support in $D$ belongs to $\cD(\cL)$, and for such a function and $x \in D$,
$$\cL f(x) = \Delta f(x) - \nabla V(x) \cdot \nabla f(x).$$
Denote by 
$$\gamma(dx) := e^{-V(x)}dx.$$
$\gamma$ is therefore one reversible measure for $\cL$. Note that $\gamma$ is not necessarily defined as a probability measure. In all what follows, it will be assumed that
$$\gamma(D) < + \infty.$$
\subsection{Proof of Theorem \ref{premiere}}
In this subsection, we will prove Theorem \ref{premiere} stated earlier in the introduction, that we recall below :
\begin{theorem}
\label{thm}
Let $(X_t)_{t \geq 0}$ following \eqref{sed2} and such that $\gamma(D) < + \infty$. \begin{enumerate}[({BE}1)]
    \item Assume that there exists a nonnegative function $\eta \in \cD(\cL)$ defined on $D \cup \d D$, positive on $D$ and vanishing on $\d D$, such that $\gamma(\eta^2) < + \infty$ and there exists $\lambda_0 > 0$ such that $$\Delta \eta - \nabla V \cdot \nabla \eta = - \lambda_0 \eta.$$ 
    \item Assume moreover that there exists $\kappa > 0$ such that
\begin{equation}
    \label{bakry}
    \Hess [V - 2 \log(\eta)] \geq \kappa \Id.
\end{equation}
\end{enumerate}
Then $\gamma(\eta) < + \infty$, the probability measure $\alpha := \eta \circ \gamma$ is a quasi-stationary distribution for $(X_t)_{t \geq 0}$ and 
\begin{enumerate}[(i)]
\item there exists a constant $C > 0$ such that, for any $\mu \in \cM_1(D)$ and $t \geq 0$,
\begin{equation}
\label{3}
    \left\| \P_\mu(X_t \in \cdot | \tau_\d > t) - \alpha \right\|_{TV} \leq C \frac{\sqrt{\gamma(D)\gamma(\eta^2)}}{\gamma(\eta)}\chi_2(\eta \circ \mu | \eta \circ \alpha) e^{-\kappa t},
\end{equation}
\item If moreover there exists $x_0 \in D$ such that
$$\int_D (1+d(x,x_0))^2e^{-V(x)}dx < + \infty,$$
then for any $\mu \in \cM_1(D)$, there exists $t_\mu$ such that for any $t \geq t_\mu$,
\begin{equation}
\label{4}
    \cW_1(\P_\mu[X_t \in \cdot | \tau_\d > t], \alpha) \leq C(\gamma, \eta) \chi_2(\eta \circ \mu|\eta \circ \alpha) e^{- \kappa t},
\end{equation}
where 
$$C(\gamma, \eta) := \left[a + b \frac{\int_D (1+d(x,x_0))\eta(x)\gamma(dx)}{\gamma(\eta)}\right]\frac{ \sqrt{\gamma(\eta^2)\int_D (1+d(x,x_0))^2 \gamma(dx)}}{\gamma(\eta)}.$$
\end{enumerate}
\end{theorem}

\begin{remark}
\label{r}
Following Remark $1$, it is possible to state general assumptions on the potential $V$ to entail (BE1). For example, when $D = (0,+\infty)^d$, the condition 
\begin{equation}\label{c}\lim_{|x| \to \infty} \nabla V(x) = + \infty\end{equation}
entails Assumption (F) in \cite{CV2017c} (see this paper for the proof), so entails (P1) and (BE1).
\end{remark}

\begin{proof} 
First of all, remark that the property $\gamma(\eta) < + \infty$ comes from the Cauchy-Schwarz inequality and the fact that $\gamma(\eta^2) \lor \gamma(D) < + \infty$ by assumptions. \\
To prove that $\alpha := \eta \circ \gamma$ is a quasi-stationary distribution for $X$, the proof of \cite[Theorem 1.1]{CMSM} will be adapted to general Kolmogorov diffusion processes. Denote by 
$$L:= \Delta - \nabla V \cdot \nabla.$$
By Itô's formula, for any $f$ twice continuously differentiable with compact support in $D$,
$$e^{\lambda_0 t} f(X_{t \land \tau_\d}) = f(X_0) + \int_0^{t \land \tau_\d} (L f(X_{s \land \tau_\d}) + \lambda_0 f(X_{s \land \tau_\d}))ds + \cM_t,$$
where $\cM$ is a martingale. Hence, for any $x \in D$ and $t \geq 0$,
\begin{equation}\label{perdu}e^{\lambda_0 t} \E_x[f(X_{t \land \tau_\d})] = f(x) + \E_x\left(\int_0^{t \land \tau_\d} (L f(X_{s \land \tau_\d}) + \lambda_0 f(X_{s \land \tau_\d}))ds\right).\end{equation}
Since $L$ is symmetric with respect to $\gamma$ (i.e. $\int_{\R^d} g Lhd\gamma = \int_{\R^d} Lg h d\gamma, \forall g,h \in \L^2(\gamma)$), by Fubini's theorem, for any $s \geq 0$,
$$\E_{\eta \circ \gamma}[Lf(X_{s \land \tau_\d}) + \lambda_0 f(X_{s \land \tau_\d})] = 0.$$
Hence, integrating \eqref{perdu} over $\eta \circ \gamma$ and using Fubini's theorem,
$$e^{\lambda_0 t} \E_{\eta \circ \gamma}[f(X_{t \land \tau_\d})] = \eta \circ \gamma(f).$$
As a result, it is shown that, for any $f$ twice continuously differentiable with compact support in $D$,
$$(\eta \circ \gamma) P_tf = e^{-\lambda_0 t} \eta \circ \gamma(f).$$
Since $\eta$ vanishes at the boundary of $D$, the probability measure $\alpha = \eta \circ \gamma$ is a quasi-stationary distribution, associated to $\lambda_0$. Moreover, remark that the measure $\gamma(dx) = e^{-V(x)}dx$ is a reversible measure for the semi-group $(P_t)_{t \geq 0}$, which means that, for any $f,g \in \cB(D)$,
$$\int_D (P_tf)g d\gamma = \int_D f(P_tg)d\gamma,~~~~\forall t \geq 0.$$ 
Then, for any $t \geq 0$ and $f \in \cB(D)$,
\begin{align*}
    \int_D (P_t \eta)fd\gamma = \int_D \eta(P_tf)d\gamma = \gamma(\eta) \alpha P_t f = e^{- \lambda_0 t} \gamma(\eta) \alpha(f) = e^{-\lambda_0 t} \int_D \eta f d\gamma.
\end{align*}
Thus, $\eta$ is also an eigenfunction for $(P_t)_{t \geq 0}$, associated to $\lambda_0$.  
Now, consider again the Doob transform $(\tilde{P}_t)_{t \geq 0}$ defined by
$$\tilde{P}_t f(x) := e^{\lambda_0 t} \frac{P_t[\eta \times f](x)}{\eta(x)},~~~~\forall x \in D, \forall f \in \cB(D).$$
Then the generator of the semi-group $(\tilde{P}_t)_{t \geq 0}$, denoted by $\tilde{\cL}$, endowed with its domain $\cD(\tilde{\cL})$, is 
\begin{align*}
    \tilde{\cL}f(x) 
 &= \Delta f(x) - \nabla \left[V - 2\log(\eta)\right](x) \cdot \nabla f(x),~~~~\forall x \in D, \forall f \in \cD(\tilde{\cL}). 
\end{align*}
The condition \eqref{bakry} is therefore the Bakry-\'{E}mery condtion for the generator $\tilde{\cL}$. This implies therefore (see for example \cite[Proposition 4.8.1]{bakry2013analysis}) that the invariant measure $\beta := \eta \circ \alpha = \eta^2 \circ \gamma$ for the semi-group $(\tilde{P}_t)_{t \geq 0}$ satisfies a Poincaré inequality with $C_P = \frac{1}{\kappa}$, that is
$$\Var_\beta(f) \leq - \frac{1}{\kappa} \int_D f \tilde{\cL}f d\beta,$$
which is (P2). 

In order to deal with the total variation distance, it is enough to take $\psi = 1$. For such a choice of $\psi$, one has
$$\eta \circ \gamma\left(\frac{\psi^2}{\eta/\alpha(\eta)}\right) = \eta \circ \gamma(\eta) \times \eta \circ \gamma(1/\eta) 
 = \frac{\gamma(D)\gamma(\eta^2)}{\gamma(\eta)^2} < \infty.$$
Hence the condition (P3) of the Theorem \ref{thm-poin} is satisfied for $\psi = 1$. So, by Theorem \ref{thm-poin}, one has 
\begin{align*}
    \| \P_\mu(X_t \in \cdot | \tau_\d > t) - \eta \circ \gamma \|_{TV} &= \sup_{|f| \leq 1} \left| \E_\mu[f(X_t)|\tau_\d > t] - \eta \circ \gamma(f) \right| \\ &\leq C \frac{\sqrt{\gamma(D)\gamma(\eta^2)}}{\gamma(\eta)}\chi_2(\eta \circ \mu | \eta^2 \circ \gamma) e^{-\kappa t}.
\end{align*} 
The point $(ii)$ of Theorem \ref{thm} is a straightforward consequence of Corollary \ref{cor1}.
\end{proof}

\begin{remark}
Exponential decays like \eqref{3} and \eqref{4} hold also under weaker assumptions than \eqref{bakry}, such as the two followings:
\begin{itemize}
    \item There exists $c > 0$ and $R \geq 0$ such that for $|x| > R$,
    \begin{equation}
    \label{sat}
        x \cdot \nabla [V - 2 \log(\eta)](x) \geq c |x|.
    \end{equation}
    \item There exists $a \in (0,1)$, $c > 0$ and $R \geq 0$ such that for $|x| > R$, 
    $$a |\nabla [V - 2 \log(\eta)](x)|^2 - \Delta [V - 2 \log(\eta)](x) > c.$$
\end{itemize}
These two conditions actually appear in \cite[Corollary 1.6]{BBCG2008} and imply (P2). In particular, the first condition is satisfied when $V - 2 \log(\eta)$ is convex. It will be shown later that, for diffusion processes on $(0,+\infty)^d$, the convexity of $V$ implies the one of $V - 2 \log(\eta)$ for a particular eigenfunction $\eta$ (which is not unique a priori), so \eqref{sat} is satisfied.
\end{remark}

\subsection{Two examples}
\label{ex}

In this subsection, two examples whose the eigenfunctions $\eta$ can be explicitly computed are studied through Theorem \ref{thm}: a scaled Brownian motion living in a hypercube, and an Ornstein-Uhlenbeck process living on $(0, + \infty)^d$. In particular, several spectral properties will be claimed throughout this subsection. We refer the reader to the Appendix, at the end of the paper, for a few proofs on these spectral properties.

\subsubsection{Brownian motion in a hypercube}
Concerning quasi-stationarity for Brownian motion living in the interior of a general compact set in $\R^d$ and absorbed at its boundary, we refer the reader to \cite[Theorem 1.1]{CMSM}.

Consider the open set $D = C_N := (-N,N)^d$, with $N \in \N$, and $V = 0$, that is to say $(X_t)_{t \geq 0} = (\sqrt{2}B_t)_{t \geq 0}$. 
Then the function $\eta_{Bm}$ defined by
\begin{equation}\label{eta}\eta_{Bm}(x_1, \ldots, x_d) := \prod_{i=1}^d \cos\left(\frac{\pi}{2N}x_i\right),~~~~\forall (x_1, \ldots, x_n) \in C_N,\end{equation}
is an eigenfunction of $\Delta$ with respect to the eigenvalue $-\lambda_0$, where 
$$\lambda_0 = \frac{d \pi^2}{4N^2}.$$
A proof of this claim is written in Appendix, at the end of the paper. Moreover, $\eta_{Bm}$ is positive on $C_N$, vanishing at $\d C_N$ and 
\begin{align*}
    \gamma(\eta_{Bm}^2) &= \int_{C_N} \eta_{Bm}^2(x)dx
    = \left( \int_{-N}^N \cos^2\left(\frac{\pi}{2N}x\right)\right)^d = N^d.
\end{align*}
Thus (BE1) is satisfied. Now, for any $(x_1, \ldots, x_d) \in C_N$ and $i,j = 1, \ldots, d$, 
$$(\Hess \log(\eta_{Bm}(x_1, \ldots, x_d)))_{i,j} = \left\{
      \begin{array}{cc}
        -\left(\frac{\pi}{2N}\right)^2 \left[1 + \tan^2\left(\frac{\pi}{2N}x_i\right)\right] &\text{ if } i=j\\
        0 &\text{ otherwise}\\
      \end{array}
    \right.$$
Hence, the Bakry-\'{E}mery condition \eqref{bakry} in (BE2) holds for $\kappa = \frac{\pi^2}{2N^2}$. Then, Theorem \ref{thm} entails that the probability measure 
$$\alpha_{Bm}(dx) = \eta_{Bm} \circ \gamma(dx) = \frac{\eta_{Bm}(x)dx}{\gamma(\eta_{Bm})} =  \left(\frac{\pi}{4N}\right)^d \prod_{i=1}^d \cos\left(\frac{\pi}{2N}x_i\right)dx$$
is a quasi-stationary distribution for $(X_t)_{t \geq 0}$ and 
So, there exists $C > 0$ such that, for any initial measure $\mu \in \cM_1(C_N)$ and $t \geq 0$,
\begin{multline}
    \label{e}
    ||\P_\mu(X_t \in \cdot | \tau_\d > t) - \alpha_{Bm}||_{TV} \leq C \frac{\sqrt{\gamma(D) \gamma(\eta_{Bm}^2)}}{\gamma(\eta_{Bm})} \chi_2(\eta_{Bm}\circ \mu | \eta_{Bm}\circ \alpha_{Bm}) \exp\left(-\frac{\pi^2}{2N^2} t\right).
\end{multline}



Now, one has
$$\gamma(\eta_{Bm}) = \left(\int_{-N}^N \cos\left(\frac{\pi}{2N}x\right)dx\right)^d = \left(\frac{4N}{\pi}\right)^d, ~~~~\text{ and }~~~~\gamma(D) = (2N)^d.$$
As a result, \eqref{e} becomes: for any $\mu \in \cM_1(C_N)$ and $t \geq 0$,
\begin{equation*}
    ||\P_\mu(X_t \in \cdot | \tau_\d > t) - \alpha_{Bm}||_{TV} \leq C \left(\frac{\pi}{2\sqrt{2}}\right)^d \chi_2(\eta_{Bm}\circ \mu | \eta_{Bm}\circ \alpha_{Bm}) \exp\left(-\frac{\pi^2}{2N^2} t\right).
\end{equation*}
Note however that this Bakry-\'{E}mery coefficient $\kappa$ is not optimal. In particular, denoting $\beta_{Bm} := \eta_{Bm}\circ \alpha_{Bm}$, if we use directly Theorem \ref{thm-poin}, the Poincaré constant $C_P$ is equal to
\begin{equation} \label{trou-spectral}1/C_P = \inf_{f \in \L^2(\beta_{Bm}), \beta_{Bm}(f) = 0} \frac{- \int_E f \tilde{\cL} f d \beta_{Bm}}{\int_E f^2 d \beta_{Bm}}.\end{equation}
By this formula, one can compute (see Appendix to see the computation) that
$$\frac{1}{C_P} = \frac{3 \pi^2}{4 N^2} > \kappa.$$
Thus, by Theorem \ref{thm-poin}, for any $\mu \in \cM_1(E)$ and $t \geq 0$,
$$\| \P_\mu[X_t \in \cdot | \tau_\d > t] - \alpha_{Bm}\|_{TV} \leq C \left(\frac{\pi}{2\sqrt{2}}\right)^d \chi_2(\eta_{Bm}\circ \mu | \beta_{Bm}) \exp\left(-\frac{3 \pi^2}{4 N^2}t\right).$$
Note then, as mentioned in Remark \ref{petcou}, the Poincaré constant $C_P$ and the Bakry-Emery constant $\kappa$ depend on the domain $D$ through the size $N$.

Concerning the $1$-Wasserstein distance, one can remark that the exponential decay in total variation distance \eqref{e} implies the one in $\cW_1$. As a matter of fact, since we are studying a process living on the compact set $(-N,N)^d$, one has, for any $\mu, \nu \in \cM_1(C_N)$,
$$\cW_1(\mu, \nu) \leq d N \| \mu - \nu \|_{TV}.$$
This inequality allows actually to get a better estimate for the decay in $1$-Wasserstein distance than the one provided by Corollary \ref{cor1}. 

Finally, note that, if the initial measure $\mu$ admits a density function with respect to Lebesgue's measure which vanishes at $\d C_N$, $\mu(\eta_{Bm}) > 0$ and, since $\alpha_{Bm}(dx) = \frac{\eta_{Bm}(x) dx}{\gamma(\eta_{Bm})}$, $$\int_{C_N} \frac{d \mu}{d \alpha_{Bm}} \eta_{Bm} d\mu = \gamma(\eta_{B_m}) \int_{C_N} \left(\frac{d\mu}{dx}(x)\right)^2 dx < + \infty,$$
so that $\chi_2(\eta_{Bm} \circ \mu | \beta_{Bm}) < + \infty$ according to Proposition \ref{domain-of-attraction-lemma}. If $\mu = \delta_x$, one can show that the probability measure $\P_\mu[X_1 \in \cdot | \tau_\d > 1]$ admits a density function with respect to Lebesgue's measure which vanishes at $\d C_N$, so one has also  $$\chi_2(\eta_{Bm} \circ  \P_\mu[X_1 \in \cdot | \tau_\d > 1] | \beta_{Bm}) < + \infty.$$ 
By Remark \ref{domain-of-attraction}, this shows that the convergence of $\P_\mu[X_t \in \cdot | \tau_\d > t]$ to $\alpha_{Bm}$ at rate $\frac{3 \pi^2}{4 N^2}$ holds for any initial distributions $\mu$.

\subsubsection{Ornstein-Uhlenbeck process}
For this example, consider $D = (0,+ \infty)^d$ and $V(x) = \frac{\lambda}{2} \sum_{i=1}^d x_i^2$, with $\lambda > 0$. Then, $(X_t)_{t \geq 0}$ is a $d$-dimensional Ornstein-Uhlenbeck process satisfying the following stochastic differential equation
\begin{equation} \label{OU}dX_t = \sqrt{2}dB_t - \lambda X_t dt,~~~~X_t \in (0,+\infty)^d.\end{equation}
A positive eigenfunction of $\Delta - \nabla V \cdot \nabla$ is \begin{equation} \label{etaou}\eta_{OU}(x_1, \ldots, x_n) := \prod_{i=1}^d x_i,~~~~\forall (x_1, \ldots, x_d) \in D,\end{equation}
associated to the eigenvalue $- \lambda d$ (the computation is in Appendix). Noting that, for this example, 
$$\gamma(dx) = \prod_{i=1}^d e^{-\frac{\lambda x_i^2}{2}}dx,$$
one has furthermore
$$\gamma(\eta_{OU}^2) = \left(\int_0^\infty x^2 e^{-\frac{\lambda x^2}{2}}dx\right)^d = \left(\sqrt{\frac{\pi}{2\lambda^3}}\right)^d.$$
Thus, (BE1) in Theorem \ref{thm} is satisfied. 

For any $x = (x_1, \ldots, x_d)$ and $i,j = 1, \ldots, d$,
$$[\Hess(V - 2 \log(\eta_{OU}))(x)]_{i,j}= \left\{
      \begin{array}{cc}
       \lambda + \frac{2}{x_i^2} &\text{ if } i=j\\
        0 &\text{ otherwise}\\
      \end{array}
    \right.$$
So the Bakry-\'{E}mery condition \eqref{bakry} is satisfied for $\kappa = \lambda$.
 Hence, by Theorem \ref{thm}, the probability measure 
 $$\alpha_{OU}(dx) := \eta_{OU} \circ \gamma(dx) = \frac{\eta_{OU}(x) \gamma(dx)}{\gamma(\eta_{OU})} = \lambda^d \prod_{i=1}^d x_i e^{-\frac{\lambda x_i}{2}}dx$$
 is a quasi-stationary distribution for $(X_t)_{t \geq 0}$ there exists $C_d > 0$ such that, for any $\mu$ and $t$ large enough,
$$\cW_1(\P_\mu[X_t \in \cdot | \tau_\d > t], \alpha_{OU}) \leq C_d \chi_2(\eta_{OU} \circ \mu | \eta_{OU} \circ \alpha_{OU}) \exp\left(- \lambda t\right).$$
Note that the rate of convergence does not depend on the dimension $d$, but the constant $C_d$ explodes in high dimension. More precisely, after computations, one can show that, when ${d \to + \infty}$, 
$$C_d \sim \frac{d(d-1)}{4 \lambda} \left(\frac{\pi}{2}\right)^d.$$

Contrary to the previous example, the probability measure $\P_\mu[X_t \in \cdot | \tau_\d > t]$ does not converge to $\alpha_{OU}$ for any initial distribution $\mu$. The curious reader can read the paper \cite{LSM2000}, where it is shown, in the one-dimensional case, that there exists an infinity of quasi-stationary distributions for Ornstein-Uhlenbeck processes absorbed by $D^c$, each associated with their own domain of attraction. It is then expected that the property 
$$\chi_2(\eta_{OU} \circ \mu | \eta_{OU}\circ \alpha_{OU}) < + \infty$$
cannot be satisfied for every initial distributions. However, $\alpha_{OU}$ has the property that, for any $x \in D$, $\P_x[X_t \in \cdot | \tau_\d > t] \underset{t \to \infty}{\longrightarrow} \alpha_{OU}$ (this can be seen for example in \cite{CV2017c}). Such a quasi-stationary distribution is called a \textit{Yaglom limit} (see for example \cite{MV2012} for the definition). Our aim is then to show that there exists $t_0 > 0$ such that
\begin{equation} \label{control}\chi_2(\eta_{OU} \circ \phi_{t_0}(\delta_x) | \eta_{OU} \circ \alpha_{OU}) < + \infty,~~~~\forall x \in D.\end{equation}
To do so, we will use Corollary \ref{attraction}, stating that this holds when $\eta_{OU}(x) > 0$ for all $x \in D$, which is satisfied, and when 
$$\int_D \frac{d \phi_{t_0}(\delta_x)}{d \alpha_{OU}} \eta_{OU} d \phi_{t_0}(\delta_x) < + \infty.$$
For sake of simplicity, let us only deal with the case $d = 1$ (the result in general dimension could be deduced by tensorization). In this case, for any $x \in (0,+\infty)$, one claims that the positive measure $\delta_x P_1 = \P_x[X_1 \in \cdot, \tau_\d > 1]$ admits a density function with respect to the measure $\sqrt{\frac{2\lambda}{\pi}}\1_{x \in D}\gamma(dx)$, denoted by $f_x$, which admits the following representation 
\begin{equation} \label{representation} f_x = \sum_{n \text{ odd}} e^{-\lambda n} F_n(x) F_n,~~~~~~\text{ in } \L^2\left(\gamma\right),\end{equation}
with
$$F_n(x) = \frac{H_n(\sqrt{\lambda} x)}{\sqrt{n!}},~~~~\forall n \in \Z_+, \forall x \in (0,+ \infty),$$
where $(H_n)_{n \in \Z_+}$ are Hermite polynomials. See the Appendix for the definition of Hermite polynomials and the proof of this claim.  

Now, let us prove \eqref{control} for $t_0 = 1$. Since $\P_x[\tau_\d > 1] > 0$ for all $x \in D$, \eqref{control} is equivalent to 
$$\int_0^\infty \frac{\delta_x P_1}{d \alpha_{OU}} \eta_{OU} d (\delta_x P_1) < + \infty.$$
Thus, using that $\delta_x P_1(dy) = \sqrt{\frac{2\lambda}{\pi}} f_x(y) \gamma(dy)$ and $\alpha_{OU}(dy) = \lambda \eta_{OU}(y) \gamma(dy)$, \eqref{control} is equivalent to 
$$\int_0^\infty f^2_x(y) \sqrt{\frac{2\lambda}{\pi}} \gamma(dy) < + \infty.$$
Then, using the representation \eqref{representation} and the fact that $(F_n)_{n \in \N}$ is orthonormal in $\L^2\left(\sqrt{\frac{2 \lambda}{\pi}} \gamma\right)$ (this is proved in the Appendix),
\begin{align*}\int_0^\infty f^2_x(y) \sqrt{\frac{2\lambda}{\pi}} \gamma(dy) &= \sum_{n \text{ odd}} e^{-2 \lambda n} F_n^2(x) \\
&\leq \sum_{n \in \Z_+} e^{-2 \lambda n} F_n^2(x).
\end{align*}
However, the representation $\sum_{n \in \Z_+} e^{-\lambda n} F_n(x) F_n$ is no less than the density function, with respect to $\sqrt{\frac{\lambda}{2\pi}} \gamma$, of the marginal law at time $1$ of an Ornstein-Uhlenbeck process of parameter $\lambda$, starting at $x$ (this fact is also proved in the Appendix, see \eqref{conclusion}). In other terms, for any $y \in \R$, 
$$\sum_{n \in \Z_+} e^{-\lambda n} F_n(x) F_n(y) \propto \frac{e^{-\frac{\lambda(y - xe^{-\lambda})^2}{1 - e^{-2 \lambda}}}}{e^{-\lambda \frac{y^2}{2}}},~~~~~~\text{ in } \L^2(\gamma).$$
This entails that, for any $x \in (0,+\infty)$,
$$\sum_{n \in \Z_+} e^{-2 \lambda n} F_n^2(x) \propto \int_\R \frac{e^{-\frac{2\lambda(y - xe^{-\lambda})^2}{1 - e^{-2 \lambda}}}}{e^{-\lambda \frac{y^2}{2}}} dy < + \infty.$$
Tu sum up, \eqref{control} holds for $t_0 = 1$ and for any $x \in (0,+\infty)$, which entails that 
$$\chi_2(\eta_{OU} \circ \phi_1(\delta_x) | \eta_{OU} \circ \alpha_{OU}) < + \infty,~~~~\forall x \in D.$$

\subsection{Diffusion processes on $(0, \infty)^d$}
\label{dif}

In general, contrary to the two previous examples, the eigenfunction $\eta$ cannot be explicitly given, so the assumptions of Theorem \ref{thm} cannot be checked in practice. In this subsection, one will see how to bypass this problem for diffusion processes living on $D = (0,+\infty)^d$ and absorbed when one of its component reaches $0$.

\subsubsection{When $d = 1$}
Take a one-dimensional diffusion process following
\begin{equation}
\label{cdfi}
dX_t = \sqrt{2} dB_t -  V'(X_t) dt \end{equation}
living on $D = (0,+\infty)$ and absorbed at $\d = 0$, where $V$ is a $\cC^2$-function. Then, one gets the following proposition. 


\begin{proposition}
\label{lemma}
Assume that $V$ is convex on $(0,+\infty)$ and 
$$\lim_{x \to + \infty} V'(x) = + \infty.$$
Then there exists an eigenfunction $\eta$ such that $\log(\eta)$ is concave.
\end{proposition}
\begin{proof}
In \cite[Corollary 4.2.]{CV2017c}, it is shown that, under the condition $\lim_{x \to + \infty} V'(x) = + \infty$, there exists a unique positive eigenfunction $\eta \in \cC^2(D)$ for $P_t$, for all $t \geq 0$, such that $P_t \eta = e^{-\lambda_0 t}$ and
\begin{equation}
\label{eigenfunction}
\eta''(x) - V'(x) \eta'(x) = - \lambda_0 \eta(x),~~~~\forall x \in D,\end{equation}
with $\lambda_0 > 0$, and such that there exists $C, \theta > 0$ such that, for any $x \in D$ and $t \geq 0$, 
\begin{equation}
\label{convergence}
\left| \eta(x) -  e^{\lambda_0 t} \P_x(\tau_\d > t) \right| \leq C e^{- \theta t} \varphi^{1/p}(x),\end{equation}
where $p > 1$ and $\varphi$ is a Lyapunov function such that there exists $D_0 \subset (0,+ \infty)$, $C' > 0$ and $\lambda > 0$ large enough such that
$$\varphi''(x) - V'(x) \varphi'(x) \leq - \lambda \varphi(x) + C' \1_{x \in D_0}, ~~~~\forall x \in D.$$
For any $x \geq 0$, $h > 0$ and $t \geq 0$,
\begin{align*}
    \P_{x+h}(\tau_\d > t) &= \P_{x+h}(\tau_\d > t, \tau_x \leq t) + \P_{x+h}(\tau_x > t) \\ &= \E_{x+h}[\1_{\tau_x \leq t} \P_x(\tau_\d > t - u)|_{u = \tau_x}] + \P_{x+h}(\tau_x > t),
\end{align*}
where $\tau_x$ is the hitting time of $x$ by the process $(X_t)_{t \geq 0}$, and where the strong Markov property is used for the second equality.

Considering the process $(X_{t \land \tau_x})_{t \geq 0}$ absorbed at $x$, it is also a diffusion process coming down from infinity. So there exists also a positive function $\eta_x$ on $(x,+ \infty)$ and a positive constant $\lambda_x$ such that, for any $y > x$, 
$$\eta_x(y) = \lim_{t \to \infty} e^{\lambda_x t} \P_y(\tau_x > t).$$
Since $\tau_0$ dominates stochastically $\tau_x$, $\lambda_0 < \lambda_x$ for any $x > 0$, so 
$$\lim_{t \to + \infty} e^{\lambda_0 t} \P_{x+h}(\tau_x > t) = 0.$$ 
Now remark that for any $x \geq 0$, $h > 0$ and $t \geq 0$,
$$e^{\lambda_0 t} \E_{x+h}[\1_{\tau_x \leq t} \P_x(\tau_\d > t - u)|_{u = \tau_x}] = \E_{x+h}[\1_{\tau_x \leq t} e^{\lambda_0 \tau_x} \times [e^{\lambda_0 (t-u)} \P_x(\tau_\d > t - u)]_{u = \tau_x}].$$
The random variable $\1_{\tau_x \leq t} e^{\lambda_0 \tau_x} \times [e^{\lambda_0 (t-u)} \P_x(\tau_\d > t - u)]_{u = \tau_x}$ is upper bounded by  $e^{\lambda_0 \tau_x}[ \eta(x) + C \varphi^{1/p}(x)]$, and, for $h$ small enough, $\E_{x+h}(e^{\lambda_0 \tau_x}) < \infty$ (see \cite[Proposition 3]{MV2012}). Moreover, it converges to $\eta(x) e^{\lambda_0 \tau_x}$ when $t$ goes to infinity. 
So, by the Lebesgue's theorem, 
$$\lim_{t \to + \infty} e^{\lambda_0 t} \E_{x+h}[\1_{\tau_x \leq t} \P_x(\tau_\d > t - u)|_{u = \tau_x}] = \E_{x+h}[e^{\lambda_0 \tau_x}] \eta(x).$$
In conclusion, one has
$$\eta(x+h) = \lim_{t \to \infty} e^{\lambda_0 t} \P_{x+h}(\tau_\d > t) = \E_{x+h}[e^{\lambda_0 \tau_x}] \eta(x).$$
So, for any $h > 0$ small enough,
$$\frac{\eta(x+h)-\eta(x)}{h} = \eta(x) \frac{\E_{x+h}[e^{\lambda_0 \tau_x}]-1}{h}.$$ 
Then, since $\eta \in \cC^2((0,+\infty))$, for any $x > 0$, $\lim_{h \downarrow 0} \frac{\E_{x+h}[e^{\lambda_0 \tau_x}]-1}{h}$ exists and 
$$\eta'(x) = \eta(x) \lim_{h \downarrow 0} \frac{\E_{x+h}[e^{\lambda_0 \tau_x}]-1}{h}.$$
In other words, one has
$$\log(\eta)'(x) =  \lim_{h \downarrow 0} \frac{\E_{x+h}[e^{\lambda_0 \tau_x}]-1}{h}.$$ 
Now, for $h$ fixed, since $V$ is convex, the derivative $V'$ is non-decreasing and, by \cite[Theorem 1.1, Chapter VI, p.437]{IW1989}, the function $x \mapsto \E_{x+h}(e^{\lambda_0 \tau_x})$ is non-increasing, so one has for any $x \leq x'$,
$$ \lim_{h \downarrow 0} \frac{\E_{x+h}[e^{\lambda_0 \tau_x}]-1}{h} \geq  \lim_{h \downarrow 0} \frac{\E_{x'+h}[e^{\lambda_0 \tau_{x'}}]-1}{h}.$$
So the function $\log(\eta)'$ is non-increasing, which implies that the function $\log(\eta)$ is concave. 
\end{proof}
The previous proposition actually tells us that, assuming $V$ convex, the second derivative of $V - 2 \log(\eta)$ is greater than the one of $V$. In particular, Proposition \ref{lemma} entails the following corollary:
\begin{corollary}
\label{corollary}
Let $(X_t)_{t \geq 0}$ satisfying \eqref{cdfi} and assume that there exists $\kappa > 0$ such that
$$V''(x) \geq \kappa,~~~~\forall x \in (0, + \infty).$$
Then there exists a quasi-stationary distribution $\alpha$, which is absolutely continuous with respect to $\gamma$, and a constant $C > 0$ such that, for any $\mu \in \cM_1(D)$ and for $t$ large enough,
$$\| \P_\mu[X_t \in \cdot | \tau_\d > t] - \alpha \|_{TV} \leq  C \chi_2(\eta \circ \mu | \eta \circ \alpha) e^{- \kappa t},$$
and 
$$\cW_1(\P_\mu[X_t \in \cdot | \tau_\d > t], \alpha) \leq C \chi_2(\eta \circ \mu | \eta \circ \alpha) e^{- \kappa t},$$
where $\eta := \frac{d \alpha}{d \gamma}$. 
\end{corollary}

\begin{proof}
 Integrating twice the Bakry-\'{E}mery condition, there exists two constants $a_1,a_2 \in \R$ such that, for any $x > 0$, 
\begin{equation}
\label{minor}
V(x) \geq a_1 + a_2x + \frac{\kappa}{2} x^2.\end{equation}
Hence, one has $\lim_{x \to + \infty} V'(x) = + \infty$ and there exists an eigenfunction $\eta$ satisfying \eqref{eigenfunction} and \eqref{convergence}. Moreover, since $V$ is convex, $\log(\eta)$ is concave by Proposition \ref{lemma}, so for any $x > 0$,
$$(V - 2 \log(\eta))''(x) \geq V''(x) \geq \kappa,$$
which implies also that $\gamma(\eta^2) < + \infty$.
Hence the conditions (BE1)-(BE2) hold. Finally, by \eqref{minor},
$$\int_0^\infty (1+x)^2 e^{-V(x)} dx < + \infty,$$
which entails the exponential decay in total variation and $1$-Wasserstein distance by Theorem \ref{thm}, setting $\alpha := \eta \circ \gamma$.
\end{proof}
\subsubsection{One-dimensional processes coming down from infinity}
Let $(X_t)_{t \geq 0}$ be a solution of \eqref{cdfi} \textit{coming down from infinity}, which means that there exists a constant $\rho > 0$ such that $\sup_{x \geq 0} \E_x(e^{\rho \tau_\d}) < + \infty$ (see \cite{bansaye2017diffusions} for alternative definitions). Quasi-stationarity for such processes have been already studied in \cite{CV2017b}, in particular $(X_t)_{t \geq 0}$ absorbed at $0$ admits a unique quasi-stationary distribution $\alpha$ absolutely continuous with respect to $\gamma$ and an eigenfunction $\eta$, unique up to a multiplicative constant, satisfying the following relation (see \cite[Theorem 4.1.]{CV2017b}):  
\begin{equation}
    \label{rmk}\eta(x) = 4 \lambda_0 \int_0^\infty (x \land y) \eta(y) \gamma(dy),\end{equation}
where $-\lambda_0 < 0$ is the eigenvalue associated to $\alpha$ and $\eta$. Moreover, \cite[Proposition 4.2.]{CV2017b} states that $\eta$ is proportional to the function
\begin{equation*}
    x \mapsto \int_0^\infty (x \land y) \alpha(dy).
\end{equation*}
In particular, $\log(\eta)$ is concave, whatever the convexity of the potential $V$.
For these processes, one can state the following result :
\begin{theorem}
\label{thm-cdfi}
Let $(X_t)_{t \geq 0}$ following \eqref{cdfi} coming down from infinity such that 
$$\tilde{\kappa} := \inf_{x > 0} \left\{ V''(x) + 8 \lambda_0 e^{-V(x)} \right\} > 0.$$
Then there exists a constant $C > 0$ such that, for any $\mu \in \cM_1(D)$ and for $t$ large enough,
\begin{equation}
    \label{1}
    \| \P_\mu[X_t \in \cdot | \tau_\d > t] - \alpha \|_{TV} \leq  C \chi_2(\eta \circ \mu | \eta \circ \alpha) e^{- \tilde{\kappa} t},
\end{equation}
and 
\begin{equation}
    \label{2}
    \cW_1(\P_\mu[X_t \in \cdot | \tau_\d > t], \alpha) \leq C \chi_2(\eta \circ \mu | \eta \circ \alpha) e^{- \tilde{\kappa} t},
\end{equation}
If moreover $V'(x) > 0$ for any $x > 0$, then the previous statement holds for
$$\tilde{\kappa} := \inf_{x > 0} \left\{  V''(x) + 8 \lambda_0 e^{-V(x)} + 2 \lambda_0^2 \left(\frac{1 - 4 e^{-V(x)}}{V'(x)}\right)^2 \right\}.$$
\end{theorem}
\begin{remark}
In other words, this theorem states that the rate of convergence $\kappa$ coming from  the Bakry-\'{E}mery condition $V'' \geq \kappa$ can actually be improved replacing it by $\tilde{\kappa}$. Moreover, this entails that the exponential convergences \eqref{1} and \eqref{2} holds even if $V$ is concave in a neighborhood of $0$, as soon as the function $x \mapsto  V''(x) + 8 \lambda_0 e^{-V(x)}$ is lower-bounded by a positive constant.
\end{remark}
\begin{proof}[Proof of Theorem \ref{thm-cdfi}]
The idea is simply to apply Theorem \ref{thm} and to compute the best $\kappa$ satisfying 
$$[V - 2 \log(\eta)]''(x) \geq \kappa,~~~~\forall x > 0,$$
knowing \eqref{rmk}. First of all, for any $x > 0$,
$$[V - 2 \log(\eta)]''(x) = V''(x) - 2 \left[ \frac{\eta''(x)}{\eta(x)} - \left(\frac{\eta'(x)}{\eta(x)}\right)^2\right].$$
By the equality \eqref{rmk}, 
$$\eta(x) = 4 \lambda_0 \int_0^x y \eta(y) \gamma(dy) + 4 \lambda_0 x\int_x^\infty \eta(y) \gamma(dy).$$
Then, for any $x > 0$,
$$\eta'(x) = 4 \lambda_0 \int_x^\infty \eta(y) \gamma(dy),~~~~\eta''(x) = - 4 \lambda_0 \eta(x)e^{-V(x)}.$$
Hence, for any $x > 0$, 
$$[V - 2 \log(\eta)]''(x) = V''(x) + 8 \lambda_0 e^{-V(x)} + 2 \left(\frac{\eta'(x)}{\eta(x)}\right)^2.$$
As a result, assuming $\tilde{\kappa} := \inf_{x > 0} \left\{ V''(x) + 8 \lambda_0 e^{-V(x)} \right\} > 0$, one has
$$[V - 2 \log(\eta)]''(x) \geq \tilde{\kappa},~~~~\forall x \in (0,+\infty).$$
Now, assuming moreover $V'(x) > 0$ for any $x > 0$, and using that $\eta''(x) - V'(x)\eta'(x) = - \lambda_0 \eta(x)$ for any $x > 0$, one has 
$$\eta'(x) = \frac{\eta''(x) +  \lambda_0 \eta(x)}{V'(x)} =  \lambda_0 \eta(x) \frac{1 - 4 e^{-V(x)}}{V'(x)},$$
which entails that 
$$[V - 2 \log(\eta)]''(x) = V''(x) + 8 \lambda_0 e^{-V(x)} + 2 \lambda_0^2 \left(\frac{1 - 4 e^{-V(x)}}{V'(x)}\right)^2,$$
which concludes the proof.
\end{proof}

\begin{example}
Considering 
$$V : x \mapsto (x+1)^\delta,~~~~\delta > 2,$$
the underlying process $(X_t)_{t \geq 0}$ satisfying \eqref{cdfi} comes down from infinity, so Theorem \ref{thm-cdfi} applies and the inequalities \eqref{1} and \eqref{2} hold for $\tilde{\kappa} := \inf_{x > 0} \left\{  V''(x) + 8 \lambda_0 e^{-V(x)} + 2 \lambda_0^2 \left(\frac{1 - 4 e^{-V(x)}}{V'(x)}\right)^2 \right\}$. For this example, the eigenvalue $-\lambda_0$ is not explicitly known, but it is possible to compare it with the eigenvalue $-\lambda_{OU}$ associated to the one-dimensional absorbed Ornstein-Uhlenbeck process satisfying
$$dX^{OU}_t = \sqrt{2} dB_t - (X_t^{OU}+1)dt,$$
and such that, for any $x > 0$,
$$\lambda_{OU} = - \lim_{t \to \infty} \frac{\log P^{OU}_t[\1_{(0,+\infty)}](x)}{t},$$
where $(P^{OU}_t)_{t \geq 0}$ is the sub-Markovian semi-group associated to $(X^{OU}_t)_{t \geq 0}$.
This eigenvalue is explicitly known : 
$$\lambda_{OU} = 1.$$
Likewise one has, for any $x > 0$,
$$\lim_{t \to \infty} - \frac{\log \P_x[\tau_\d > t]}{t} = \lambda_0.$$
Hence, since $V'(x) \geq x+1$ for any $x > 0$, one deduces from the theorem of comparison \cite[Theorem 1.1, Chapter VI]{IW1989} that, for any $x > 0$, 
$$\lambda_0 \geq \lambda_{OU} = 1.$$
As a result, one has a lower-bound for $\lambda_0$ and one can choose $\tilde{\kappa}$ as 
$$\tilde{\kappa} := \inf_{x > 0} \left\{  V''(x) + 8 e^{-V(x)} + 2 \left(\frac{1 - 4 e^{-V(x)}}{V'(x)}\right)^2 \right\}.$$
\end{example}

\subsubsection{Multi-dimensional diffusion processes}

Now consider one $d$-dimensional diffusion process $(X_t)_{t \geq 0} := (X^1_t, \ldots, X^d_t)_{t \geq 0}$ satisfying 
$$dX_t =  \sqrt{2} d B_t -  \nabla V(X_t) dt$$
where, for any $x_1, \ldots, x_d \in (0,+\infty)^d$,
\begin{equation}
\label{decomposition}
V(x_1, \ldots, x_d) = \sum_{i=1}^d V_i(x_i),\end{equation}
where, for any $1 \leq i \leq d$, $V_i : [0, + \infty) \to \R$ is a convex $\cC^2([0,+\infty))$-function such that $\lim_{x \to \infty} V'_i(x) = + \infty$. We consider this process as  absorbed by the boundary of $[0,+\infty)^d$. In particular, $D = (0,+\infty)^d$ and $$\d D = \{(x_1, \ldots, x_d) \in [0,+\infty)^d : x_i = 0 \text{ for some } i\}.$$  
Denote $\eta$ a common nonnegative eigenfunction of $(P_t)_{t \geq 0}$. Then, for any $x_1, \ldots, x_d$, $\eta$ can be expressed as follows  
$$\eta(x_1, \ldots, x_d) := \prod_{i=1}^d \eta_i(x_i),$$
where $(\eta_i)_{i = 1, \ldots, d}$ are functions such that, for any $i$, there exists $\lambda_{0,i} > 0$ such that, for any $t \geq 0$ and $x \in (0,+\infty)$,
$$P^i_t \eta_i(x) = e^{- \lambda_{0,i} t} \eta_i(x),$$
where 
$$P_t^i f(x) = \E_x[f(X^i_t)\1_{\tau^i_\d > t}],~~~~\forall f \in \cB((0,+\infty)), \forall t \geq 0,$$
where $\tau^i_\d := \inf\{t \geq 0 : X^i_t = 0\}$.
$\eta$ is therefore associated to $\lambda_0 := \sum_{i=1}^d \lambda_{0,i}$, and one has for any $x_1, \ldots, x_d \in (0,+ \infty)^d$
$$(\Hess \log(\eta)(x_1, \ldots, x_d))_{i,j} = \left\{ \begin{array}{cc}
\log(\eta_i(x_i))'' & \text{ if } i=j \\
0 & \text{ otherwise.}
\end{array}
\right.$$
By what it was shown previously, for any $i = 1, \ldots, d$, $\log(\eta_i)$ is concave. As a result, one can state the following result, which is the the multi-dimensional version of Corollary \ref{corollary}, already stated in the Introduction. 
\begin{theorem}
\label{thmintro}
Assume that the potential can be written as \eqref{decomposition} and that there exists $\kappa > 0$ such that, for any $i = 1, \ldots, d$,
$$V_i''(x) \geq \kappa,~~~~\forall x \in [0,+\infty).$$
Then there exists a quasi-stationary distribution $\alpha := \eta \circ \gamma$ and a constant $C_d > 0$ (depending on the dimension $d$) such that, for any $\mu \in \cM_1(D)$ and $t$ large enough, 
$$\cW_1(\P_\mu[X_t \in \cdot | \tau_\d > t], \alpha) \leq C_d \chi_2(\eta \circ \mu | \eta \circ \alpha) e^{-\kappa t},$$
and 
$$\|\P_\mu[X_t \in \cdot | \tau_\d > t] - \alpha\|_{TV} \leq C_d \chi_2(\eta \circ \mu | \eta \circ \alpha) e^{-\kappa t}.$$     
\end{theorem}
Previously, it was seen, with the two examples of Subsection \ref{ex}, that the constant $C_d$ could explode when the dimension $d$ goes to infinity. However, it is possible to improve this result when the initial measure $\mu$ is the tensorial product of $d$ probability measures on $(0,+\infty)$. In this case,
since \eqref{decomposition} is assumed, the one-dimensional processes $(X^i)_{i = 1, \ldots, d}$ are mutually independent.
Moreover, since $\{X_t \ne 0\} = \bigcap_{i = 1, \ldots, d} \{X^i_t \ne 0\}$, then for any $t \geq 0$ and $\mu_1, \ldots, \mu_d \in \cM_1((0,+\infty))$,
$$\P_{\mu_1 \otimes \cdots \otimes \mu_d}[X_t \in \cdot | \tau_\d > t] = \P_{\mu_1}[X^1_t \in \cdot | \tau^1_\d > t] \otimes \cdots \otimes \P_{\mu_d}[X^d_t \in \cdot | \tau^d_\d > t],$$
 Then, one obtains the following theorem, which was also stated previously in the Introduction.
\begin{theorem}
\label{dernier}
Assume the assumptions of Theorem \ref{thmintro}. Then there exists a constant $C > 0$, which does not depend on the dimension, such that, for any $\mu_1, \ldots, \mu_d \in \cM_1((0,+\infty))$, and for $t$ large enough, 
$$\|\P_{\mu_1 \otimes \cdots \otimes \mu_d}[X_t \in \cdot | \tau_\d > t] - \alpha\|_{TV} \leq C \left[ \sum_{i=1}^d \chi_2(\eta_i \circ \mu_i | \eta_i \circ \alpha_i)\right] e^{-\kappa t},$$
and
$$\cW_1(\P_{\mu_1 \otimes \cdots \otimes \mu_d}[X_t \in \cdot | \tau_\d > t], \alpha) \leq C \left[ \sum_{i=1}^d \chi_2(\eta_i \circ \mu_i | \eta_i \circ \alpha_i)\right] e^{-\kappa t},$$
where $\alpha_i(dx) := \eta_i(x) e^{-V_i(x)}dx$.
\end{theorem}
\begin{proof}
The first result comes from the inequalities
$$\| \mu_1 \otimes \cdots \otimes \mu_d - \nu_1 \otimes \cdots \otimes \nu_d \|_{TV} \leq \sum_{i=1}^d \| \mu_i - \nu_i\|_{TV},$$
which can be shown using the equality 
$$\frac{1}{2}\| \mu - \nu \|_{TV} = \inf_{(X,Y) \in \Pi(\mu,\nu)} \P(X \ne Y),~~~~\forall \mu, \nu \in \cM_1(D),$$
and the result is deduced from the one obtained for $d=1$.
In the same way, by the definition of $\cW_1$ and recalling that $\cW_1$ is defined through the $L^1$-distance defined in \eqref{distance}, one has
$$\cW_1(\mu_1 \otimes \cdots \otimes \mu_d, \nu_1 \otimes \cdots \otimes \nu_d) = \sum_{i=1}^d \cW_1(\mu_i, \nu_i),$$
which implies the second inequality in the statement of Theorem \ref{dernier}.  
\end{proof}
Obvioulsy, one can also state a result similar to Theorem \ref{thm-cdfi} for multi-dimensional diffusion processes coming down from infinity: 
\begin{theorem}
\label{cdfimulti}
Assume that, for any $i=1,\ldots,d$, $X~i$ comes down from infinity and $V_i'(x) > 0$ for any $x > 0$. Then the statements of Theorem \ref{thmintro} and \ref{dernier} hold replacing $\kappa$ by 
$$\tilde{\kappa} := \min_{i=1,\ldots,d} \inf_{x > 0} \left\{ V_i''(x) + 8 \lambda_{0,i} e^{-V_i(x)} + 2 \lambda_{0,i}^2 \left(\frac{1 - 4e^{-V_i(x)}}{V_i'(x)}\right)^2 \right\}.$$
\end{theorem}
\textbf{Acknowledgement.} I am very grateful to the anonymous referee for his/her comments and questions, which allow me to better the paper. This research was supported by the
Swiss National Foundation grant 200020 196999. 

\section*{Appendix}

\subsection*{Spectral analysis for the Brownian motion in a hypercube}

This part is dedicated to the spectral analysis of the scaled Brownian motion $(X_t)_{t \geq 0} = (\sqrt{2}B_t)_{t \geq 0}$ living in the hypercube $C_N := (-N,N)^d$. 

For any $k_1, \ldots, k_n \in \N$, let $\eta_{k_1, \ldots, k_d}$ defined by
\begin{equation*}\eta_{k_1, \ldots, k_n}(x_1, \ldots, x_d) = \frac{1}{\sqrt{N^d}} \prod_{i=1}^d \sin\left(\frac{k_i \pi}{2N}(x_i + N)\right),~~~~\forall (x_1, \ldots, x_d) \in C_N,\end{equation*}
These functions are therefore eigenfunctions of the Laplacian $\Delta$: for all $k_1, \ldots, k_d \in \N$, for any $(x_1, \ldots, x_d) \in C_N$, 
\begin{align}
    \cL \eta_{k_1, \ldots, k_d}(x_1, \ldots, x_d) = \Delta \eta_{k_1, \ldots, k_d}(x_1, \ldots, x_d) &= \frac{1}{\sqrt{N^d}} \sum_{i=1}^d \d_{x_i,x_i} \prod_{i=1}^d \sin\left(\frac{k_i \pi}{2N}(x_i + N)\right) \notag \\
    &= - \frac{1}{\sqrt{N^d}} \sum_{i=1}^d \left(\frac{k_i \pi}{2N}\right)^2 \prod_{i=1}^d \sin\left(\frac{k_i \pi}{2N}(x_i + N)\right) \notag \\
    &= - \lambda_{k_1, \ldots, k_d} \eta_{k_1, \ldots, k_d}(x_1, \ldots, x_d),
    \label{ex2}
\end{align}
where 
$$\lambda_{k_1, \ldots, k_d} := \frac{\pi^2}{4 N^2} \sum_{i=1}^d k_i^2.$$
In particular, by the definition of the function $\eta_{Bm}$ in \eqref{eta}, $$\eta_{Bm} = \sqrt{N^d} \eta_{1, \ldots, 1}.$$ Hence, by \eqref{ex2},
$$\Delta \eta_{Bm} = - \lambda_0 \eta_{Bm},$$
where 
$$\lambda_0 := \lambda_{1,\ldots,1} = \frac{d \pi^2}{4 N^2}.$$
Hence, $\eta_{Bm}$ is indeed an eigenfunction of $\Delta$, as claimed in Subsection \ref{ex}. 

It remains us to compute the Poincaré constant $C_P$. 
The family $(\eta_{k_1, \ldots, k_d})_{k_1, \ldots, k_d}$ is a total orthonormal basis of $\L^2(\gamma)$ (recalling that $\gamma$ is Lebesgue's measure in our case) and are eigenvectors of $\Delta$, by \eqref{ex2}. Then $(\eta_{k_1, \ldots, k_d})_{k_1, \ldots, k_d}$ are also eigenvectors of $P_t$, for all $t$, respectively associated to the eigenvalues $(e^{-\lambda_{k_1, \ldots, k_d} t})_{k_1, \ldots, k_d}$.
Then, defining for any $k_1, \ldots, k_d \in \N$
$$\tilde{\eta}_{k_1, \ldots, k_d} = \frac{\eta_{k_1, \ldots, k_d}}{\eta_{Bm}},$$
one obtains, for any $t \geq 0$,
$$\tilde{P}_t \tilde{\eta}_{k_1, \ldots, k_d} = \frac{e^{\lambda_0 t}}{\eta_{Bm}} P_t[\eta_{k_1, \ldots, k_d}] = e^{-\left(\lambda_{k_1, \ldots, k_d} - \lambda_0\right) t} \tilde{\eta}_{k_1, \ldots, k_d}.$$
Since the family $(\eta_{k_1, \ldots, k_d})_{k_1, \ldots, k_d}$ is orthonormal with respect to Lebesgue's measure, the family $(\tilde{\eta}_{k_1, \ldots, k_d})_{k_1, \ldots, k_d}$ is orthogonal with respect to the measure $\beta_{Bm} = \eta_{Bm} \circ \alpha_{Bm}$. Thus, the family $(\frac{\tilde{\eta}_{k_1, \ldots, k_d}}{\|\tilde{\eta}_{k_1, \ldots, k_d}\|_{\L^2(\beta_{Bm})}})_{k_1, \ldots, k_d}$ is a total orthonormal basis of $\L^2(\beta_{Bm})$ and one obtains by \eqref{trou-spectral} that
$$1/C_P = \lambda_1 - \lambda_0,$$
where $\lambda_1$ is the smallest $\lambda_{k_1, \ldots, k_d}$ different from $\lambda_0$. In other words,
$$\lambda_1 = \lambda_{1, \ldots, 1, 2} = \frac{(d-1)\pi^2}{4N^2} + \frac{\pi^2}{N^2},$$ so that
$$\frac{1}{C_P} = \lambda_1 - \lambda_0 = \frac{3 \pi^2}{4 N^2} > \kappa,$$
which is exactly the claim stated in Subsection \ref{ex}.

\subsection*{Spectral analysis of the Ornstein-Uhlenbeck process living in $(0,+\infty)^d$}

Let us consider the $d$-dimensional process $(X_t)_{t \geq 0}$ defined by 
$$dX_t = \sqrt{2} dB_t - \lambda X_t dt,$$
with $\lambda > 0$, living on $(0,\infty)^d$ and absorbed at its boundary. First, we prove that the function $\eta_{OU}$ defined by 
$$\eta_{OU}(x_1, \ldots, x_d) = \prod_{i=1}^d x_i$$
is a right eigenfunction for the operator $L = \Delta - \lambda x \cdot \nabla$, positive on $(0,+\infty)^d$ and vanishing at its boundary. As a matter of fact,
\begin{align*}L \eta_{OU}(x) = \Delta \eta_{OU}(x) - \lambda x \cdot \nabla \eta_{OU}(x)& = - \lambda \sum_{i=1}^d x_i \d_{x_i} \eta_{OU}(x_1, \ldots, x_d) \\
&= - \lambda d \prod_{j=1}^d x_j = - \lambda d \eta_{OU}(x_1, \ldots, x_d),\end{align*}
which proves one of the claim of Subsection \ref{ex}.

It remains us to prove the formula of representation \eqref{representation} for the density function $f_x$. Before proving this equality, let us recall some facts on Hermite polynomials:
\begin{definition}
Hermite polynomials $(H_n)_{n \in \Z_+}$ are defined as follows: for any $n \in \Z_+$, for any $x \in \R$,
$$H_n(x) = (-1)^n e^{\frac{x^2}{2}} \frac{d^n}{d^nx} e^{-\frac{x^2}{2}}.$$
\end{definition}
Then two interesting properties can be deduced from this definition:
\begin{itemize}
    \item For any $n,m \in \Z_+$,
    \begin{equation} \label{first}\int_\R H_n(x) H_m(x) \frac{1}{\sqrt{2\pi}} e^{-\frac{x^2}{2}}dx = n! \delta_{n,m},\end{equation}
    where $\delta_{n,m}$ is the Kronecker delta. 
    \item For any $n \in \Z_+$ and $x \in \R$,
    \begin{equation}\label{second}H_n''(x) - x H_n'(x) = - n H_n(x).\end{equation}
\end{itemize}
In particular, by the first property, denoting $\gamma_0(dx) = \frac{1}{\sqrt{2\pi}}e^{-\frac{x^2}{2}}$, the family $(H_n/\sqrt{n!})_{n \in \Z_+}$ is an orthonormal basis of $\L^2(\gamma_0)$. 
Consider now the family $(F_n)_{n \in \Z_+}$ defined by 
\begin{equation}\label{fn}F_n(x) = \frac{H_n(\sqrt{\lambda}x)}{\sqrt{n!}},~~~~\forall n \in \N, \forall x \in (0,+\infty).\end{equation}
Then, by \eqref{first}, 
\begin{equation} \label{first-fn}\int_\R F_n(x) F_m(x) \sqrt{\frac{\lambda}{2 \pi}} e^{-\lambda \frac{x^2}{2}}dx = \int_\R \frac{H_n(\sqrt{\lambda}x)}{\sqrt{n!}} \frac{H_m(\sqrt{\lambda}x)}{\sqrt{m!}} \frac{1}{\sqrt{2 \pi}} e^{-\frac{(\sqrt{\lambda} x)^2}{2}} \sqrt{\lambda} dx = \delta_{n,m}.\end{equation}
Now, by \eqref{second},
\begin{align}
    F_n''(x) - \lambda x F_n'(x) &= \frac{\lambda H''_n(\sqrt{\lambda}x) - \lambda \times \sqrt{\lambda} x H'_n(\sqrt{\lambda}x)}{\sqrt{n!}} \notag \\
    &= - \lambda n F_n(x). \label{second-fn}
\end{align}
Hence, by \eqref{first-fn}, the family $(F_n)_{n \in \Z_+}$ is an orthonormal basis of $\L^2(\tilde{\gamma})$, where 
$$\tilde{\gamma}(dx) := \sqrt{\frac{\lambda}{2 \pi}} e^{-\frac{\lambda x^2}{2}}dx = \sqrt{\frac{\lambda}{2 \pi}} \gamma(dx).$$
Furthermore, by \eqref{second-fn}, for any $n \in \N$, $F_n$ is a right eigenfunction for the Ornstein-Uhlenbeck semi-group $(S_t)_{t \geq 0}$, defined by 
$$S_tf(x) = \E\left[f\left(xe^{-\lambda t} + \sqrt{\frac{1-e^{-2 \lambda t}}{\lambda}} Z\right)\right],~~~~\forall f \in \cB(\R), \forall x \in \R,$$
with $Z$ be a standard Gaussian variable. More precisely, for any $n \in \Z_+$, for any $t \geq 0$,
$$S_t F_n(x) = e^{-\lambda n t} F_n(x),~~~~\forall x \in \R.$$
Now, in order to prove the claim in Subsection \ref{ex}, let us prove the following proposition:

\begin{proposition}Let $d = 1$. For any $A \subset D$ and $x \in D$,
$$P_1 \1_A(x) = \sum_{n \text{ odd}} e^{-\lambda n} \left(\int_A F_n(y) \sqrt{\frac{2 \lambda}{\pi}}e^{-\frac{\lambda y^2}{2}}dy\right) F_n(x),$$
where $(F_n)_{n \in \Z_+}$ is defined as in \eqref{fn}. In particular, for any $x \in D$, $\delta_x P_1$ admits a density function with respect to the measure $\sqrt{\frac{2 \lambda}{\pi}} e^{\frac{-\lambda y^2}{2}}dy$, denoted by $f_x$, whose a representation is 
\begin{equation} \label{new-representation}f_x = \sum_{n \text{ odd}} e^{- \lambda n} F_n(x) F_n.\end{equation}
\end{proposition}
\begin{proof}
Let $A \subset D$. In this proof, let us consider the process $(X_t)_{t \geq 0}$ as a non-absorbed process, so as an Ornstein-Uhlenbeck process living on $\R$ following 
$$dX_t = \sqrt{2}dB_t - \lambda X_t dt.$$
Since the Ornstein-Uhlenbeck process satisfies a property of reflection at $0$, one has, for any $x \in D$, 
\begin{equation} \label{reflection}P_1 \1_A(x) = \P_x[X_1 \in A, \tau_\d > 1] = \P_x[X_1 \in A] - \P_x[X_1 \in A, \tau_\d \leq 1] = \P_x[X_1 \in A] - \P_x[X_1 \in - A],\end{equation}
where $-A := \{x \in \R : -x \in A\}$. Another way to write \eqref{reflection} is 
$$P_1\1_A(x) = \E_x[(\1_A - \1_{-A})(X_1)].$$
Since $(F_n)_{n \in \Z_+}$ is a total orthonormal basis of $\L^2(\tilde{\gamma})$ which are eigenfunctions for the operator $S_1$, respectively associated to the eigenvalues $e^{-\lambda n}$, then for any $f \in \L^2(\tilde{\gamma})$,
\begin{equation} \label{conclusion}\E_x[f(X_1)] = S_1 f(x) = \sum_{n \in \Z_+} e^{- \lambda n} <F_n, f>_{\tilde{\gamma}} F_n(x),\end{equation}
where $<F_n,f>_{\tilde{\gamma}} = \int_\R F_n(y) f(y) \frac{\lambda}{\sqrt{2\pi}}e^{-\frac{\lambda y^2}{2}}dy$. 
Noting that the function $\1_A - \1_{-A}$ is odd and that $F_n$ is odd (respectively even) when $n$ is odd (respectively even), one has
$$<F_n, \1_A - \1_{-A}>_{\tilde{\gamma}} = \left\{\begin{array}{cc}
    0 & \text{ if } n \text{ is even.} \\
    \int_A F_n(y) \sqrt{\frac{2\lambda}{\pi}}e^{-\frac{\lambda y^2}{2}}dy & \text{ otherwise.} 
\end{array}\right.$$
In conclusion, using \eqref{reflection} and \eqref{conclusion},
$$P_1 \1_A(x) = \sum_{n \text{ odd}} e^{-\lambda n} \int_A F_n(y) \sqrt{\frac{2\lambda}{\pi}}e^{-\frac{\lambda y^2}{2}}dy F_n(x).$$
The representation \eqref{new-representation} is naturally deduced from the previous equality. 
\end{proof}

\bibliographystyle{plain}
\bibliography{biblio}
\end{document}